\documentclass[11pt]{article}

\usepackage[english]{babel}
\newtheorem{theorem}{Theorem}
\newtheorem{lemma}{Lemma}
\newtheorem{definition}{Definition}
\newtheorem{proof}{Proof}
\usepackage{stmaryrd}
\usepackage{enumerate}
\usepackage{algorithm}
\usepackage[noend]{algpseudocode}
\usepackage{calrsfs}
\usepackage{graphicx}
\usepackage{amssymb}
\usepackage{array}
\usepackage{url}
\usepackage{amsmath}
\usepackage{epstopdf}
\usepackage{epsfig}
\usepackage{subcaption}
\usepackage{pstricks}
\usepackage{fancyheadings}
\usepackage{pgfplots}
\usepackage{multirow}
\usepackage{physics}
\usepackage{stackengine}
\usepackage{hyphenat}
\usepackage{cleveref}
\usepackage[margin=1.0in]{geometry}
\DeclareMathAlphabet{\pazocal}{OMS}{zplm}{m}{n}

\usepackage{graphicx}

\newcommand{\lJump}{[\![}
\newcommand{\rJump}{]\!]}

\usepackage[framemethod=tikz]{mdframed}
\usepackage{lipsum}

\definecolor{mycolor}{rgb}{0.122, 0.435, 0.698}
\definecolor{blue-green}{rgb}{0.0, 0.87, 0.87}
\definecolor{royalblue}{rgb}{0.01, 0.28, 1.0}
\definecolor{bluet}{rgb}{0.1, 0.1, 0.8}

\newmdenv[innerlinewidth=1pt, roundcorner=0.5pt,linecolor=mycolor,innerleftmargin=0.5pt,
innerrightmargin=0.5pt,innertopmargin=0.4pt,innerbottommargin=0.4pt]{mybox}

\DeclareMathAlphabet{\pazocal}{OMS}{zplm}{m}{n}

\usepackage{bm}

\newcommand{\ba}{\begin{array}}
\newcommand{\ea}{\end{array}}
\newcommand{\be}{\begin{equation}}
\newcommand{\ee}{\end{equation}}
\newcommand{\ben}{\begin{equation*}}
\newcommand{\een}{\end{equation*}}
\newcommand{\bd}{\begin{displaymath}}
\newcommand{\ed}{\end{displaymath}}
\newcommand{\bi}{\begin{itemize}}
\newcommand{\ei}{\end{itemize}}
\newcommand{\bn}{\begin{enumerate}}
\newcommand{\en}{\end{enumerate}}

\newtheorem{remark}{Remark}

\begin{document}

\title{Well-posed Boundary Conditions and Energy Stable Discontinuous Galerkin Spectral Element Method for the Linearized Serre Equations}
\author{
  Kenny Wiratama\thanks{Mathematical Sciences Institute, Australian National University, Australia
  ({kenny.wiratama@anu.edu.au}).}
  \and Kenneth Duru\thanks{Mathematical Sciences Institute, Australian National University, Australia
  ({kenneth.duru@anu.edu.au}).}
  \and Stephen Roberts\thanks{Mathematical Sciences Institute, Australian National University, Australia
  ({stephen.roberts@anu.edu.au}).}
    \and Christopher Zoppou\thanks{Mathematical Sciences Institute, Australian National University, Australia
  ({christopher.zoppou@anu.edu.au}).}}
\maketitle

\begin{abstract}
We derive well-posed boundary conditions for the linearized Serre equations in one spatial dimension by using the energy method. An energy stable and conservative discontinuous Galerkin spectral element method with simple upwind numerical fluxes is proposed for solving the initial boundary value problem. We derive discrete energy estimates for the numerical approximation and prove a priori error estimates in the energy norm. Detailed numerical examples are provided to verify the theoretical analysis and show convergence of numerical errors.
\end{abstract}

\section{Introduction}
The propagation of free surface water waves is governed by the Euler equations, under the assumption that the fluid flow is inviscid and irrotational. The free surface nature makes the problem of solving the Euler equations difficult \cite{mitsotakis2017modified}. Consequently, there are approximate models derived from the Euler equations, such as the shallow water wave equations \cite{craik2004origins} and the Serre equations \cite{dutykh2013finite}. In contrast to the commonly used shallow water wave equations, the Serre equations are derived without hydrostatic pressure assumptions, and hence they contain higher order nonlinear dispersive terms. As a result, the Serre equations can model dispersive water waves \cite{zoppou2017numerical}. 

The Serre equations in one spatial dimension (1D) describing nonlinear dispersive water waves over a horizontal bed can be written as the system of partial differential equations (PDEs)
\begin{subequations}\label{serre}
\begin{align}
    \frac{\partial \zeta}{\partial t} + \frac{\partial (\bar{u}\zeta)}{\partial x} = 0, \label{eq:continuity}
    \end{align}
    \begin{align}
    \frac{\partial (\bar{u}\zeta)}{\partial t} + \frac{\partial}{\partial x} \left[\zeta\bar{u}^2 + \frac{g\zeta^2}{2} + \frac{\zeta^3}{3}\left(\left( \frac{\partial \bar{u}}{\partial x}\right)^2 - \bar{u} \frac{\partial^2 \bar{u}}{\partial x^2} - \frac{\partial ^2 \bar{u}}{\partial x\partial t}\right) \right] = 0. \label{eq:momentum}
    \end{align}
\end{subequations}
Here, $x$ is the spatial coordinate, $t$ is time, $\bar{u} = \bar{u}(x,t)$ is the depth\hyp{}averaged velocity of a free surface fluid with water depth $\zeta = \zeta(x,t)$, and $g$ is the gravitational acceleration. The continuity equation \eqref{eq:continuity} describes the conservation of mass and \eqref{eq:momentum} describes the conservation of momentum.   The conservation of momentum equation \eqref{eq:momentum} contains higher order nonlinear derivative terms, and mixed space-time derivatives of $\bar{u}$. The presence of these terms makes the Serre equations difficult to study both numerically and analytically.

In recent years, considerable research has been devoted to the development of numerical methods for solving the Serre equations. Various numerical techniques have been utilized, and they include finite difference methods \cite{do1993numerical, nwogu1993alternative, beji1996formal}, finite volume methods \cite{zoppou2017numerical, cienfuegos2006fourth, cienfuegos2007fourth, dutykh2013finite}, continuous or discontinuous Galerkin finite element methods \cite{mitsotakis2014galerkin,clamond2017conservative,ZHAO2020109729}, or combinations of these methods. Many of these approaches are restricted to being low order accurate, and are not able to efficiently resolve highly oscillatory dispersive wave modes present in the solution. To circumvent the challenge presented by higher order and mixed derivatives terms in the momentum equation \eqref{eq:momentum}, several methods introduce auxiliary variables to rewrite the Serre equations as a system of first order equations \cite{zoppou2017numerical, cienfuegos2006fourth, cienfuegos2007fourth, dutykh2013finite, ZHAO2020109729}. The main disadvantage of this approach is that we require additional constraints and non-physical boundary and interface conditions to solve the system of first order equations. Another drawback is that the degrees of freedom for the system increase by several factors, see for example \cite{ZHAO2020109729}, which requires up to 8 auxiliary variables to eliminate higher order derivatives terms in the 1D Serre equations. In two spatial dimensions, the required number of auxiliary variables is expected to increase significantly, further limiting the efficiency of the method.

The primary objective of this paper is the development of robust (provably stable), efficient (no auxiliary variables), and high order accurate numerical method for solving the Serre equations, with rigorous mathematical support. Our contributions are two-fold: 1) the derivation of well-posed boundary conditions for the linearized Serre equations; 2) the development of a provably energy stable discontinuous Galerkin spectral element method (DGSEM) for the initial boundary value problem (IBVP).

A necessary and important step towards developing a robust and high order accurate numerical scheme for solving PDEs in a bounded domain is to derive well-posed boundary conditions for the system of PDEs, in particular boundary conditions that ensure that the IBVP at the continuous level is well-posed \cite{ghader2014revisiting}. Unfortunately, the theory of IBVPs for dispersive waves such as the Serre equations is less developed  \cite{cienfuegos2007fourth, noelle2022class}. One of the main difficulties is that there are no well-defined characteristics. Therefore, the theory of characteristics often used for hyperbolic IBVPs is not applicable. In the literature,  most numerical methods for the Serre equations are derived for periodic boundary conditions, and consequently they are not relevant when non-periodic boundary phenomena are important. There are however a few exceptions \cite{cienfuegos2007fourth,noelle2022class,kazakova2020discrete}, where non-periodic boundary conditions are considered, although for specific time discretizations. In \cite{kazakova2020discrete}, non-local artificial boundary conditions for the linearized Serre equations with zero background flow velocity are proposed and analyzed. In the present work, we derive linear well-posed and energy stable boundary conditions for the linearized Serre equations with arbitrary background flow velocity.  The derived boundary conditions are local, and yield bounded energy estimates for the solutions of the IBVP. Given appropriate data, the boundary conditions can be used to effectively impose inflow and outflow boundary conditions.

We will derive a provably stable and high order accurate DGSEM for solving the IBVP without introducing auxiliary variables. For an element based scheme, one of the main challenges lies in how to connect adjacent elements, in particular enforcing the continuity of the solutions and their (first and second) derivatives across the elements' interfaces in a stable manner without introducing auxiliary variables. Another challenge is the derivation of stable and accurate numerical boundary treatment for the IBVP. To succeed, in addition to well-posed boundary conditions, we derive well-posed interface conditions that ensure the conservation of energy, mass, and linear momentum at the continuous level. At the discrete level, following \cite{duru2022conservative}, we construct spatial derivative operators that satisfy the summation by parts property (SBP) in a discontinuous Galerkin spectral element framework. Then, we use the Simultaneous Approximation Term (SAT) method \cite{carpenter1994time} to weakly impose interface and boundary conditions. This SBP-SAT approach enables us to prove that the semi-discrete numerical scheme satisfies discrete energy estimates analogous to the continuous energy estimates necessary for the well-posedness of the IBVP. The semi-discrete numerical approximation is integrated in time using the classical fourth order accurate explicit Runge-Kutta method.  Our proposed numerical scheme combines key ideas  from the SBP finite difference method, the spectral element method, and the discontinuous Galerkin method.  We perform detailed numerical experiments to verify the theoretical analysis, showing convergence of numerical errors, conservation properties of the method, and demonstrate the effectiveness of the high order numerical method in resolving highly oscillatory dispersive waves. The results obtained from the linear analysis can be applied to the nonlinear problem. This will be reported in a forthcoming work.  

 The paper is organized as follows: In section 2 we introduce integration by parts identities to be mimicked by the spatial discrete derivative operators. In section 3 the linearized Serre equations are introduced, and we  perform continuous analysis, proving well-posedness for the initial value problem (IVP) and the IBVP. In section 4 we derive the DGSEM and prove numerical stability. Discrete error estimates are derived in section 5. Numerical experiments are presented in section 6 verifying the theoretical analysis. In section 7, we draw conclusions and suggest directions for future work.

\section{Preliminaries}
We begin by introducing some notations. Let $u$ and $v$ be real-valued functions defined in an interval domain $\Omega = [x_L, x_R]$. The standard $L^2$\hyp{}inner product and its associated norm are denoted by
\begin{equation*}\label{eq:L2_inner_product}
    (u,v)_{\Omega} = \int_{\Omega} u(x)v(x) \: dx \quad \text{and} \quad \|u\|_{L^2(\Omega)}^2 = (u,u)_{\Omega}
\end{equation*}
respectively. Assuming that $u$ and $v$ are sufficiently smooth, that is $u, v \in C^{p}(\Omega)$ for some $p\ge 3$, the integration-by-parts principles, for first, second and third derivatives, yield
\begin{equation} \label{ibp_x}
    \left( u, \frac{dv}{dx} \right)_{\Omega} = \left.uv\right|^{x=x_R}_{x=x_L} - \left(\frac{du}{dx},v \right)_{\Omega},
\end{equation}
\begin{equation} \label{ibp_xx}
    \left( u, \frac{d^2v}{dx^2} \right)_{\Omega} = \left.u \frac{dv}{dx}  \right|^{x=x_R}_{x=x_L} - \left(\frac{du}{dx}, \frac{dv}{dx} \right)_{\Omega},
\end{equation}
\begin{equation} \label{ibp_xxx}
    \left( u, \frac{d^3v}{dx^3} \right)_{\Omega} = \left.u \frac{d^2v}{dx^2}\right|^{x=x_R}_{x=x_L} -\frac{1}{2} \left.\frac{du}{dx} \frac{dv}{dx}\right|^{x=x_R}_{x=x_L} +\frac{1}{2}\left(\frac{d^2u}{dx^2}, \frac{dv}{dx}\right)_{\Omega}- \frac{1}{2}\left(\frac{du}{dx}, \frac{d^2v}{dx^2}\right)_{\Omega}.
\end{equation}
In particular, if $u = v$ then \eqref{ibp_x}, \eqref{ibp_xxx} yield
\begin{equation} \label{firstorderidentity}
    \left( u, \frac{du}{dx} \right)_{\Omega} = \left.\frac{u^2}{2} \right|^{x=x_R}_{x=x_L},
\end{equation}
\begin{equation} \label{thirdorderidentity}
\left( u, \frac{d^3 u}{d x^3} \right)_{\Omega} = \left. \left(u \frac{d^2 u}{d x^2} - \frac{1}{2}\left( \frac{d u}{d x}\right)^2 \right)  \right|^{x=x_R}_{x=x_L},
\end{equation}
respectively.

The identities derived in this section are key ingredients for  the continuous and  discrete analysis performed in the next sections. In \cref{SBPsection}, we will describe how to construct spatial operators that mimic these identities at the discrete level. This will enable us to prove discrete counterparts of the results obtained in the continuous analysis.

\section{Linearized Serre equations}
\label{sec:main}
Linearizing \eqref{serre} around the constant mean height $H > 0$ and constant velocity $U$ gives the linearized Serre equations
\begin{subequations} \label{linearizedserre}
\begin{align}
    &\frac{\partial h}{\partial t} + \frac{\partial }{\partial x}\left(H  u + U  h\right) = 0, \label{linearizedserre1} \\
    &\frac{\partial u}{\partial t} + \frac{\partial }{\partial x}\left(g h + U  u - \frac{H^2U}{3}\frac{\partial^2u}{\partial x^2} - \frac{H^2}{3}\frac{\partial^2 u}{\partial x \partial t}\right) = 0, \: x \in \Omega, \: t \geq 0, \label{linearizedserre2}
\end{align}
\end{subequations}
where $h$ and $u$ denote the perturbed height and velocity respectively. 
At $t=0$, we augment \eqref{linearizedserre} with sufficiently smooth initial conditions
\begin{align}\label{eq:initial_condition}
h(x, 0) = f_h(x), \quad u(x, 0) = f_u(x),
\end{align}
where the functions $f_h$ and $f_u$ are compactly supported in $\Omega$.  

\subsection{Well-posedness of the IVP}
Let us now consider the IVP \eqref{linearizedserre}--\eqref{eq:initial_condition} and investigate the well-posedness of the model. For several problems, such as the linear shallow water equations, the systems of PDEs have the form  $q_t + D q =0 $,  where the differential  operator $D$ depends only on the spatial derivatives $\partial /\partial x$. In these settings, the semi-boundedness\footnote{The differential operator $D$ is semi-bounded in the function space $\mathbb{V}$ if  $\left(q, Dq\right)_{\Omega} \ge \alpha \|q\|^2_{L^2(\Omega)}$ for all $q \in \mathbb{V}$, where $\alpha \in \mathbb{R}$ is a constant independent of $q$.} of $D$ ensures the well-posedness of the IVP \cite{gustafsson1995time}. However, for the Serre equations \eqref{linearizedserre}, and specifically in \eqref{linearizedserre2}, the flux contains higher order and mixed space-time derivatives of $u$. It is not obvious how to eliminate the time derivative from the spatial operator in \eqref{linearizedserre2} without introducing auxiliary variables \cite{ZHAO2020109729} or other conservative variables \cite{zoppou2017numerical,pitt2022numerical,pitt2021solving,pitt2018behaviour}.

In order to prove the well-posedness of the IVP \eqref{linearizedserre}-\eqref{eq:initial_condition}, we will bound the solution in the energy norm defined by 
\begin{equation} \label{energy}
    E(t) = \frac{g}{2} \|h\|^2_{L^2(\Omega)}  + \frac{H}{2} \| u \|^2_{L^2(\Omega)}  +  \frac{H^3}{6} \left\|\frac{\partial u}{\partial x}\right\|^2_{L^2(\Omega)}.
\end{equation}
The energy norm $ E(t)$ is a quasi $H^1$-norm, and will bound the $L^2$-norm of the height $h$ and the $H^1$-norm of the velocity $u$.

We begin with the definition
\begin{definition}
The IVP \eqref{linearizedserre}-\eqref{eq:initial_condition} is well-posed if there is a unique solution satisfying the energy estimate
\begin{align*}
    E(t) \le K e^{\alpha t} E(0),
\end{align*}
    for some constants $\alpha $ and $K >0$ that are independent of the initial conditions \eqref{eq:initial_condition}, where $E(t)$ is defined in \eqref{energy}.
\end{definition}
Then, we introduce the lemma which relates the rate of change of the energy to boundary terms.
\begin{lemma} \label{theoremenergy} 
The linearized Serre equations \eqref{linearizedserre} satisfy the energy equation
\begin{equation} \label{energyequation}
    \frac{dE}{dt} = \left.\operatorname{BT}\right|^{x=x_R}_{x=x_L},
\end{equation}
where $\operatorname{BT}|^{x=x_R}_{x=x_L}=\operatorname{BT}|_{x=x_R} - \operatorname{BT}|_{x=x_L}$ and the boundary term $\operatorname{BT}$ is given by
\begin{align}
\begin{split} \label{boundaryterm}
\operatorname{BT}(h, u) = &-gHhu -\frac{gU}{2}h^2 - \frac{HU}{2}u^2 
-\frac{H^3U}{6}\left( \frac{\partial u}{\partial x}\right)^2 +\frac{H^3U}{3}\left( u\frac{\partial^2 u}{\partial x^2}\right)
+\frac{H^3}{3}\left(u\frac{\partial^2 u}{\partial x \partial t}\right).
\end{split}
\end{align}
\end{lemma}
\begin{proof}
   We multiply \eqref{linearizedserre1} by $gh$ and \eqref{linearizedserre2} by $H u $,  and integrate over $\Omega$. We have
\begin{align*}
    &\frac{g}{2}\frac{d}{dt}  \|h\|^2_{L^2(\Omega)} + gH\left( h, \frac{\partial u}{\partial x} \right)_{\Omega} + gU\left( h, \frac{\partial h}{\partial x} \right)_{\Omega} = 0,  \\
    &\frac{H}{2}\frac{d}{dt}  \|u\|^2_{L^2(\Omega)}+ \left(Hu, \frac{\partial }{\partial x}\left(g h + U  u - \frac{H^2U}{3}\frac{\partial^2u}{\partial x^2} - \frac{H^2}{3}\frac{\partial^2 u}{\partial x \partial t}\right)\right)_{\Omega} = 0. 
\end{align*}
Substituting $v = {\partial u}/{\partial t}$ in \eqref{ibp_xx} gives the identity
\begin{equation} \label{secondorderidentity}
\left(u, \frac{\partial^3 u}{\partial x^2 \partial t} \right)_{\Omega} = \left. \left( u \frac{\partial^2 u}{\partial x \partial t}\right) \right|^{x=x_R}_{x=x_L} - \frac{1}{2} \frac{d}{dt} \left\| \frac{\partial u}{\partial x} \right\|^2_{L^2(\Omega)}.
\end{equation}
Hence using \eqref{ibp_x}, \eqref{firstorderidentity}, \eqref{thirdorderidentity} and \eqref{secondorderidentity} yields
\begin{subequations}
\begin{align}
    &\frac{g}{2}\frac{d}{dt}  \|h\|^2_{L^2(\Omega)} - gH\left( \frac{\partial h}{\partial x}, u \right)_{\Omega} =  -\left.\left(gHhu +\frac{gU}{2}h^2\right) \right|^{x=x_R}_{x=x_L}, \label{linearizedserre1_proof_0} \\
&\frac{d}{dt}\left( \frac{H}{2} \| u \|^2_{L^2(\Omega)} + \frac{H^3}{6} \left\|\frac{\partial u}{\partial x}\right\|^2_{L^2(\Omega)} \right) +  gH\left(u, \frac{\partial h}{\partial x} \right)_{\Omega}
\label{linearizedserre2_proof_0}
\\
&=-\left.\left(\frac{HU}{2}u^2 + \frac{H^3U}{6}\left( \frac{\partial u}{\partial x}\right)^2-\frac{H^3U}{3}\left( u\frac{\partial^2 u}{\partial x^2}\right)-\frac{H^3}{3}\left(u\frac{\partial^2 u}{\partial x \partial t}\right)\right)\right|^{x=x_R}_{x=x_L},
\nonumber
\end{align}
\end{subequations}
where the boundary terms have been moved to the right hand side. Summing  \eqref{linearizedserre1_proof_0} and \eqref{linearizedserre2_proof_0} together completes the proof of the lemma.
\end{proof}
The following theorem proves the well-posedness of the IVP \eqref{linearizedserre}--\eqref{eq:initial_condition}.
\begin{theorem}\label{theorem:well-posednes_ivp}
Consider the linearized Serre equations \eqref{linearizedserre} with the initial conditions \eqref{eq:initial_condition} in the domain $\Omega = [x_L,x_R]$ subject to the periodic boundary conditions
\begin{align}\label{eq:peridic_boundary_conditions}
\begin{split}
   &h(x_L,t) = h(x_R,t), \quad u(x_L,t) = u(x_R,t), \\
    &\frac{\partial u}{\partial x}(x_L,t) = \frac{\partial u}{\partial x}(x_R,t), \quad \frac{\partial^2 u}{\partial x^2}(x_L,t) = \frac{\partial^2 u}{\partial x^2}(x_R,t).
\end{split}
\end{align}
The energy $E(t)$ defined in \eqref{energy} is conserved, that is 
\begin{equation*}
E(t) = E(0), \quad \forall t \ge 0. 
\end{equation*}
\end{theorem}
\begin{proof}
From Lemma \ref{theoremenergy}, we have 
   \begin{equation*} 
    \frac{dE}{dt} = \left.\operatorname{BT}\right|^{x=x_R}_{x=x_L},
\end{equation*}
where the boundary term $\operatorname{BT}$ is given by \eqref{boundaryterm}. We impose the periodic boundary conditions \eqref{eq:peridic_boundary_conditions}, and hence the boundary terms cancel out, which gives $\operatorname{BT}|^{x=x_R}_{x=x_L} =0$. It follows that 
${dE}/{dt} = 0$. Integrating in time gives $E(t) = E(0)$ for all $t \ge 0$. 
\end{proof}
Theorem \ref{theorem:well-posednes_ivp} also holds for Cauchy problems with compactly supported initial data \eqref{eq:initial_condition}.

\subsection{Well-posed boundary conditions} \label{wellposedbcs}
The main aim of this section is to formulate well-posed boundary conditions for the linearized Serre equations \eqref{linearizedserre}. This is a necessary step towards accurate and reliable computations of the solution of \eqref{linearizedserre}. Recall that for many linear PDEs, such as the linear shallow water equations, the systems have the form  $q_t + D q =0 $, where the differential operator $D$ depends only on the spatial derivative $\partial /\partial x$. In these situations, the maximally semi-boundedness\footnote{The differential operator $D$ is maximally semi-bounded if it is semi-bounded in the function space $\mathbb{V}$ but not in any space with fewer boundary conditions \cite{gustafsson1995time}.} of $D$ will ensure the well-posedness of the IBVP \cite{gustafsson1995time}.

For the linearized Serre equations \eqref{linearizedserre}, when boundary conditions are introduced, we will aim to bound the solutions in the energy norm $E(t)$. From the energy equation \eqref{energyequation}, it suffices to ensure that the boundary term is always nonpositive, namely $\left.\operatorname{BT}\right|^{x=x_R}_{x=x_L} \leq 0$. Hence, we are looking for a minimal number of boundary conditions, with homogeneous boundary data, at $x =x_L$ and $x=x_R$ such that $\left.\operatorname{BT}\right|^{x=x_R}_{x=x_L} \leq 0$.

We begin by rewriting the boundary term \eqref{boundaryterm} in a matrix form as follows:
\begin{equation*}
    \operatorname{BT} = \mathbf{v}^{\top} \mathbf{A} \mathbf{v},
\end{equation*}
where $\mathbf{v} = \left[h, u, \dfrac{\partial u}{\partial x}, \dfrac{\partial^2 u}{\partial x^2}, \dfrac{\partial^2 u}{\partial x \partial t}\right]^\top$ and $\mathbf{A}$ is the symmetric matrix defined by
\begin{equation} \label{matrixA}
    \mathbf{A} = \left[\begin{matrix}- \frac{g U}{2} & - \frac{g H}{2} & 0 & 0 & 0\\- \frac{g H}{2} & - \frac{H U}{2} & 0 & \frac{H^{3} U}{6} & \frac{H^{3}}{6}\\0 & 0 & - \frac{H^{3} U}{6} & 0 & 0\\0 & \frac{H^{3} U}{6} & 0 & 0 & 0\\0 & \frac{H^{3}}{6} & 0 & 0 & 0\end{matrix}\right].
\end{equation}
Using eigen-decomposition, see \cref{eigendecompositionappendix}, we have
\begin{equation} \label{boundarytermw}
    \operatorname{BT} = \mathbf{v}^{\top} \mathbf{A} \mathbf{v}= \mathbf{w}^{\top} \mathbf{\Lambda} \mathbf{w} = \sum_{i=1}^5 	\lambda_i {w}_i^2,
\end{equation}
where the vector $\mathbf{w}$ is given by
\begin{equation}\label{eq:vector_w}
    \mathbf{w}  =  \left[\begin{matrix}
    \dfrac{\partial^2 u}{\partial x^2}\\[1em]
    \dfrac{\partial u}{\partial x}\\[1em]
    h\\
    \dfrac{1}{C^+}\left( 2 H^{2} U \dfrac{\partial^2 u}{\partial x^2} + 2 H^{2} \dfrac{\partial^2 u}{\partial x \partial t} - 6 g h -  \left(3 U + \sqrt{4 H^{4} + 9 U^{2}}\right)u\right)\\[1em]
    \dfrac{1}{C^-}\left(2 H^{2} U \dfrac{\partial^2 u}{\partial x^2} + 2 H^{2} \dfrac{\partial^2 u}{\partial x \partial t} - 6 g h - \left(3 U - \sqrt{4 H^{4} + 9 U^{2}}\right) u\right)\end{matrix}\right]
\end{equation}
with positive constants $C^{\pm} = \sqrt{4 H^{4} + \left(3 U \pm \sqrt{4 H^{4} + 9 U^{2}}\right)^{2}}$. The corresponding eigenvalues are
\begin{align}
\begin{split}
\label{eq:eigenvalues}
 &\lambda_1 = 0, \quad  \lambda_2 = - \frac{H^{3} U}{6},  \quad \lambda_3 = - \frac{g U}{2},\\
 &\lambda_4 = - \frac{H U}{4} - \frac{H \sqrt{4 H^{4} + 9 U^{2}}}{12}, \quad \lambda_5 = - \frac{H U}{4} + \frac{H \sqrt{4 H^{4} + 9 U^{2}}}{12}.
 \end{split}
\end{align}

Without loss of generality, we only consider boundary conditions for $U = 0$ and $U > 0$. 
\subsubsection{Case 1: $U = 0$}
When $U = 0$, there are only two nonzero eigenvalues $\lambda_4 = -\lambda_5 = -\dfrac{H^3}{6}$, and hence we have
\begin{equation*} \label{boundarytermw_U_0}
     \left.\operatorname{BT}\right|^{x=x_R}_{x=x_L} = \left(\lambda_4 {w}_4^2 + \lambda_5 {w}_5^2\right)|_{x=x_R} - \left(\lambda_4 {w}_4^2 + \lambda_5 {w}_5^2\right)|_{x=x_L}.
\end{equation*}
Since $\lambda_4 < 0$ and $\lambda_5 > 0$, we need one boundary condition at $x = x_L$ and one boundary condition at $x = x_R$.
We set the linear boundary conditions
\begin{equation} \label{bczero}
    {w}_4(x_L,t) = \alpha {w}_5(x_L,t) \quad \text{and} \quad {w}_5(x_R,t) = \beta {w}_4(x_R,t),
\end{equation}
where the constants $\alpha, \beta$ must be chosen such that $\left.\operatorname{BT}\right|^{x=x_R}_{x=x_L} \le 0$ to ensure well-posedness.
\begin{lemma}\label{Lemma:BT_U0}
Consider the boundary conditions \eqref{bczero} and the boundary term $\operatorname{BT}$ defined by \eqref{boundarytermw} with $U =0$. If $-1 \leq \alpha \leq 1$ and $-1 \leq \beta \leq 1$ then $\left.\operatorname{BT}\right|^{x=x_R}_{x=x_L} \le 0$.
\end{lemma}
\begin{proof}
  Since $U = 0$ implies $\lambda_2 = \lambda_3 = 0$ and $\lambda_4 = -\lambda_5 = -\dfrac{H^3}{6}$, applying the boundary conditions \eqref{bczero} yields
\begin{align*}
\left.\operatorname{BT}\right|^{x=x_R}_{x=x_L} &= \frac{H^3}{6}\left({w}_5^2(x_R,t) - {w}_4^2(x_R,t) + {w}_4^2(x_L,t) - {w}_5^2(x_L,t)  \right) \\
&=  \frac{H^3}{6}\left((\beta^2 - 1){w}_4^2(x_R,t) + (\alpha^2 - 1){w}_5^2(x_L,t)\right) \leq 0,
\end{align*}
for all $-1 \leq \alpha \leq 1$ and $-1 \leq \beta \leq 1$.
\end{proof}
We state the following theorem which proves the well-posedness of the IBVP \eqref{linearizedserre}, \eqref{eq:initial_condition} and \eqref{bczero}.
\begin{theorem}\label{theorem:well-posednes_ibvp_u0}
Consider the linearized Serre equations \eqref{linearizedserre} with $U = 0$ subject to the initial conditions \eqref{eq:initial_condition}  and the boundary conditions \eqref{bczero}. If $-1 \leq \alpha \leq 1$ and $-1 \leq \beta \leq 1$ then 
the energy $E(t)$ defined in \eqref{energy} is bounded by the energy of the initial data, that is 
\begin{equation*}
E(t) \le E(0), \quad \forall t \ge 0.
\end{equation*}
\end{theorem}
\begin{proof}
From Lemmas \ref{theoremenergy} and \ref{Lemma:BT_U0}, if $-1 \leq \alpha \leq 1$ and $-1 \leq \beta \leq 1$, we have 
   \begin{equation*} 
    \frac{dE}{dt} = \left.\operatorname{BT}\right|^{x=x_R}_{x=x_L} \le 0.
\end{equation*}
Time integration gives 
\begin{equation*}
E(t) \le E(0), \quad \forall t \ge 0.
\end{equation*}
\end{proof}

\subsubsection{Case 2: $U > 0$}
When $U > 0$, we have $\lambda_2, \lambda_3, \lambda_4 < 0$ and $\lambda_5 > 0$. In this case, the boundary term is
\begin{align*}
 \left.\operatorname{BT}\right|_{x=x_L}^{x=x_R} =  &\left.\left( 	\lambda_2{w}_2^2 + \lambda_3{w}_3^2 + \lambda_4{w}_4^2 + \lambda_5{w}_5^2\right)\right|_{x=x_R} -  \left.\left( 	\lambda_2{w}_2^2 + \lambda_3{w}_3^2 + \lambda_4{w}_4^2 + \lambda_5{w}_5^2\right)\right|_{x=x_L}.
\end{align*}
Thus, we need three boundary conditions at the inflow boundary $x =x_L$ and one boundary condition at the outflow boundary $x =x_R$.
We set the linear boundary conditions, 
\begin{align}\label{eq:bc_Up}
\begin{split}
    &{w}_2(x_L,t) = \alpha_2 {w}_5(x_L,t), \quad {w}_3(x_L,t) = \alpha_3 {w}_5(x_L,t), \quad {w}_4(x_L,t) = \alpha_4 {w}_5(x_L,t),\\
    &{w}_5(x_R,t) = \beta_2 {w}_2(x_R,t)  + \beta_3 {w}_3(x_R,t) + \beta_4 {w}_4(x_R,t).
\end{split}
\end{align}
The constants $\alpha_j$, $\beta_j$ must be chosen such that $\left.\operatorname{BT}\right|^{x=x_R}_{x=x_L} \le 0$.
\begin{lemma}\label{Lemma:BT_Up}
Consider the boundary conditions \eqref{eq:bc_Up}, and let the symmetric matrix
\begin{equation*}\label{eq:symmetric_boundary_matrix}
\mathbf{R} = \left[\begin{matrix}
\lambda_2+\beta_2^2\lambda_5 & \beta_2\beta_3\lambda_5 & \beta_2\beta_4\lambda_5 \\
\beta_2\beta_3\lambda_5 & \lambda_3+\beta_3^2\lambda_5 & \beta_3\beta_4\lambda_5 \\
\beta_2\beta_4\lambda_5 & \beta_3\beta_4\lambda_5 & \lambda_4+\beta_4^2\lambda_5 
\end{matrix}\right].
\end{equation*}
If the constants $\alpha_j$ and $\beta_j$ (for $j = 2, 3, 4$) are chosen such that $\mathbf{v}^{\top}\mathbf{R}\mathbf{v} \le 0$ for all $\mathbf{v}\in \mathbb{R}^3$ and 
$\lambda_2 \alpha_2^2 + \lambda_3 \alpha_3^2 + \lambda_4 \alpha_4^2 + \lambda_5 \geq 0$, then $\left.\operatorname{BT}\right|^{x=x_R}_{x=x_L} \le 0$, where $\operatorname{BT}$ is the boundary term defined by $\eqref{boundarytermw}$.
\end{lemma}
\begin{proof}
  Recall that when $U > 0$, we have $\lambda_2, \lambda_3, \lambda_4 < 0$ and $\lambda_5 > 0$, and  the boundary term is
\begin{align*}
 \left.\operatorname{BT}\right|_{x=x_L}^{x=x_R} =  &\left.\left( 	\lambda_2{w}_2^2 + \lambda_3{w}_3^2 + \lambda_4{w}_4^2 + \lambda_5{w}_5^2\right)\right|_{x=x_R} \\
 &- \left.\left( 	\lambda_2{w}_2^2 + \lambda_3{w}_3^2 + \lambda_4{w}_4^2 + \lambda_5{w}_5^2\right)\right|_{x=x_L}.
\end{align*}
Applying the boundary conditions \eqref{eq:bc_Up} gives
\begin{align*}
    \left.\operatorname{BT}\right|^{x=x_R}_{x=x_L} = &\left.\mathbf{v}^{\top}\mathbf{R}\mathbf{v}\right|_{x=x_R} -\left(\lambda_2 \alpha_2^2 + \lambda_3 \alpha_3^2 + \lambda_4 \alpha_4^2 + \lambda_5\right) w_5^2(x_L,t), 
\end{align*}
where $\mathbf{v} = \left[{w}_2, {w}_3, {w}_4\right]^{\top}$.
Thus if $\mathbf{v}^{\top}\mathbf{R}\mathbf{v} \le 0$ and 
$\lambda_2 \alpha_2^2 + \lambda_3 \alpha_3^2 + \lambda_4 \alpha_4^2 + \lambda_5 \geq 0$, then we have $\left.\operatorname{BT}\right|^{x=x_R}_{x=x_L} \le 0$.
\end{proof}
As in the previous case, the following theorem shows the well-posedness of the IBVP \eqref{linearizedserre}, \eqref{eq:initial_condition} and   \eqref{eq:bc_Up}.
\begin{theorem}\label{theorem:well-posednes_ibvp_UP}
Consider the linearized Serre equations \eqref{linearizedserre} with $U > 0$ subject to the initial conditions \eqref{eq:initial_condition} and the boundary conditions \eqref{eq:bc_Up}. If the constants $\alpha_j$ and $\beta_j$ (for $j = 2, 3, 4$) are chosen such that $\mathbf{v}^{\top}\mathbf{R}\mathbf{v} \le 0$ for all $\mathbf{v}\in \mathbb{R}^3$ and 
$\lambda_2 \alpha_2^2 + \lambda_3 \alpha_3^2 + \lambda_4 \alpha_4^2 + \lambda_5 \geq 0$, then the energy $E(t)$ defined in \eqref{energy} is bounded by the energy of the initial data, that is 
\begin{equation*}
E(t) \le E(0), \quad \forall t \ge 0.    
\end{equation*}
\end{theorem}
\begin{proof}
  Using Lemmas \ref{theoremenergy} and \ref{Lemma:BT_Up}, if $\mathbf{v}^{\top}\mathbf{R}\mathbf{v} \le 0$ for all $\mathbf{v}\in \mathbb{R}^3$ and 
$\lambda_2 \alpha_2^2 + \lambda_3 \alpha_3^2 + \lambda_4 \alpha_4^2 + \lambda_5 \geq 0$, we obtain 
\begin{equation*} 
    \frac{dE}{dt} = \left.\operatorname{BT}\right|^{x=x_R}_{x=x_L} \le 0.
\end{equation*}
Time integration gives 
\begin{equation*}
E(t) \le E(0), \quad \forall t \ge 0.    
\end{equation*}
The proof is complete.
\end{proof}

\begin{remark}
When $U < 0$, the situation reverses since $x = x_L$ becomes the outflow boundary and $x = x_R$ becomes the inflow boundary. The signs of the eigenvalues also change, that is $\lambda_2, \lambda_3, \lambda_5 > 0$ and $\lambda_4 < 0$. The boundary term is
\begin{align*}
\left.\operatorname{BT}\right|_{x=x_L}^{x=x_R} = &\left.\left( 	\lambda_2{w}_2^2 + \lambda_3{w}_3^2 + \lambda_4{w}_4^2 + \lambda_5{w}_5^2\right)\right|_{x=x_R} \\
    &- \left.\left( 	\lambda_2{w}_2^2 + \lambda_3{w}_3^2 + \lambda_4{w}_4^2 + \lambda_5{w}_5^2\right)\right|_{x=x_L}.
\end{align*}
We need one boundary condition at the outflow boundary $x = x_L$ and three boundary conditions at the inflow boundary $x = x_R$. Similar to the previous case where $U >0$, we set the linear boundary conditions
\begin{align}\label{eq:bc_Un}
\begin{split}
    &{w}_4(x_L,t) = \beta_2 {w}_2(x_L,t)  + \beta_3 {w}_3(x_L,t) + \beta_5 {w}_5(x_R,t),\\
    &{w}_2(x_R,t) = \alpha_2 {w}_4(x_R,t), \quad {w}_3(x_R,t) = \alpha_3 {w}_4(x_R,t), \quad {w}_5(x_R,t) = \alpha_5 {w}_4(x_R,t),
\end{split}
\end{align}
where the constants $\alpha_j$ and $\beta_j$ must be chosen such that $\left.\operatorname{BT}\right|^{x=x_R}_{x=x_L} \le 0$. 
Carrying out a similar analysis as in \textbf{Case 2} will prove the well-posedness of the corresponding IBVP, \eqref{linearizedserre}--\eqref{eq:initial_condition} and \eqref{eq:bc_Un}. 
\end{remark}

\subsection{Well-posed interface conditions} \label{sec:interfaceconditions}
We now derive well-posed interface conditions that will be used to couple adjacent elements together. These interface conditions should enable conservative and stable numerical treatments. 

We begin by splitting the spatial domain $\Omega$ into two subdomains $\Omega^{-}$ and $\Omega^{+}$ with an interface at $x = 0$. In particular, we have $\Omega = \Omega^{-} \bigcup \Omega^{+}$, where $\Omega^{-} = [x_L, 0]$, $\Omega^{+} = [0, x_R]$, $x_L < 0$ and $x_R > 0$. The solutions of the linearized Serre equations in the subdomains $\Omega^{-}$ and $\Omega^{+}$ are denoted with the superscripts $-$ and $+$ respectively. Hence, we have
\begin{subequations} \label{linearizedserre_minus}
\begin{align}
    &\frac{\partial h^{-}}{\partial t} + \frac{\partial }{\partial x}F_h(h^{-}, u^{-}) = 0, \label{linearizedserre1_minus} \\
    &\frac{\partial u^{-}}{\partial t} + \frac{\partial }{\partial x}F_u(h^{-}, u^{-}) = 0, \: x \in \Omega^{-}, \: t \geq 0, \label{linearizedserre2_minus}
\end{align}
\end{subequations}
\begin{subequations} \label{linearizedserre_plus}
\begin{align}
    &\frac{\partial h^{+}}{\partial t} + \frac{\partial }{\partial x}F_h(h^{+}, u^{+}) = 0, \label{linearizedserre1_plus} \\
    &\frac{\partial u^{+}}{\partial t} + \frac{\partial }{\partial x}F_u(h^{+}, u^{+}) = 0, \: x \in \Omega^{+}, \: t \geq 0, \label{linearizedserre2_plus}
\end{align}
\end{subequations}
where the flux functions are given by
\begin{align*}
    F_h(h, u) = Uh + Hu, \quad F_u(h, u) = g h + U  u - \frac{H^2U}{3}\frac{\partial^2u}{\partial x^2} - \frac{H^2}{3}\frac{\partial^2 u}{\partial x \partial t}.
\end{align*}

At the interface $x =0$, we define the jump of a scalar/vector field $v$ across the interface by
$$
\lJump v \rJump: = v^+ - v^-, \quad x = 0.
$$
We are now looking for a minimal number of conditions connecting the problems \eqref{linearizedserre_minus}-\eqref{linearizedserre_plus} across the interface such that the resulting coupled problem is conservative and energy stable.
\subsubsection{Conservative interface conditions}
Let $\phi_h, \phi_u \in C^{\infty}(\Omega)$ be smooth functions in $\Omega$. We multiply \eqref{linearizedserre1_minus} and \eqref{linearizedserre2_minus} with $g\phi_h$ and $H\phi_u$ respectively, and integrate over $\Omega^-$. We have
{
$$
\left(g \phi_h,  \frac{\partial h^{-}}{\partial t}\right)_{\Omega^{-}} +  \left(g \phi_h,  \frac{\partial }{\partial x}F_h(h^{-}, u^{-})\right)_{\Omega^{-}} = 0, \quad \left(H \phi_u,  \frac{\partial u^{-}}{\partial t}\right)_{\Omega^{-}} +   \left(H \phi_u,  \frac{\partial }{\partial x}F_u(h^{-}, u^{-})\right)_{\Omega^{-}} = 0.
$$
}
Integration by parts gives
{
\begin{align}
&\left(g \phi_h,  \frac{\partial h^{-}}{\partial t}\right)_{\Omega^{-}} -  \left(g \frac{\partial \phi_h }{\partial x},  F_h(h^{-}, u^{-})\right)_{\Omega^{-}} =-\left.\phi_{h}gF_h(h^{-}, u^{-})\right|_{x=x_L}^{x=0}, \label{hminus}\\
&\left(H \phi_u,  \frac{\partial u^{-}}{\partial t}\right)_{\Omega^{-}} -   \left(H \frac{\partial  \phi_u}{\partial x},  F_u(h^{-}, u^{-})\right)_{\Omega^{-}} = -\left.\phi_{u}HF_u(h^{-}, u^{-})\right|_{x=x_L}^{x=0}. \label{uminus}
\end{align}
}
Applying the same procedure to \eqref{linearizedserre1_plus} and \eqref{linearizedserre2_plus} yields
{
\begin{align}
&\left(g \phi_h,  \frac{\partial h^{+}}{\partial t}\right)_{\Omega^{+}} -  \left(g \frac{\partial \phi_h }{\partial x},  F_h(h^{+}, u^{+})\right)_{\Omega^{+}} =-\left.\phi_{h}gF_h(h^{+}, u^{+})\right|_{x=0}^{x=x_R}, \label{hplus} \\
&\left(H \phi_u,  \frac{\partial u^{+}}{\partial t}\right)_{\Omega^{+}} -   \left(H \frac{\partial  \phi_u}{\partial x},  F_u(h^{+}, u^{+})\right)_{\Omega^{+}} = -\left.\phi_{u}HF_u(h^{+}, u^{+})\right|_{x=0}^{x=x_R}. \label{uplus}
\end{align}
}
Summing equations \eqref{hminus} and \eqref{hplus} gives
\begin{equation*}
\left(g \phi_h,  \frac{\partial h}{\partial t}\right)_{\Omega}  -  \left(g \frac{\partial \phi_h }{\partial x},  F_h(h, u)\right)_{\Omega} = \phi_{h}g\lJump F_h(h, u) \rJump   - \left.\phi_{h}gF_h(h, u)\right|_{x=x_L}^{x=x_R},
\end{equation*}
where we have collected the interface terms in the right hand side. Similarly, adding up \eqref{uplus} and \eqref{uminus} gives
\begin{equation*}
\left(H \phi_u,  \frac{\partial u}{\partial t}\right)_{\Omega}  -  \left(H \frac{\partial  \phi_u}{\partial x},  F_u(h, u)\right)_{\Omega}
=  \phi_{u}H\lJump F_u(h, u) \rJump  - \left.\phi_{u}HF_u(h, u)\right|_{x=x_L}^{x=x_R}.
\end{equation*}
Thus, if we require that the jump of the flux functions vanish, that is 
$$
\lJump F_h(h, u) \rJump = 0, \quad \lJump F_u(h, u) \rJump = 0,
$$
we obtain 
\begin{equation*}
\begin{split}
&\left(g \phi_h,  \frac{\partial h}{\partial t}\right)_{\Omega}  -  \left(g \frac{\partial \phi_h }{\partial x},  F_h(h, u)\right)_{\Omega} - 
=  - \left.\phi_{h}gF_h(h, u)\right|_{x=x_L}^{x=x_R}, \\
&\left(H \phi_u,  \frac{\partial u}{\partial t}\right)_{\Omega} - \left(H \frac{\partial  \phi_u}{\partial x},  F_u(h, u)\right)_{\Omega}
=    -\left.\phi_{u}HF_u(h, u)\right|_{x=x_L}^{x=x_R}.
\end{split}
\end{equation*}
If $\left.F_h(h, u)\right|_{x=x_L}^{x=x_R} = 0$ and $\left.F_u(h, u)\right|_{x=x_L}^{x=x_R} = 0$, choosing $\phi_{h} = \phi_{u} = 1$ yields
\begin{equation*}
\left(g,   \frac{\partial h}{\partial t}\right)_{\Omega} = 0, \quad \left(H, \frac{\partial u}{\partial t}\right)_{\Omega} = 0.
\end{equation*}
These equations imply that the total mass $\left(g,   h\right)_{\Omega}$ and the total linear momentum $\left(H, u\right)_{\Omega}$ are conserved.
\begin{theorem}\label{theorem:conservation_principle}
Consider the Serre equations \eqref{linearizedserre_minus}--\eqref{linearizedserre_plus} in the split domain $\Omega = \Omega^{-} \bigcup \Omega^{+}$, and assume that the jump of the flux functions vanish at the interface $x=0$, that is
$$
\lJump F_h(h, u) \rJump = 0, \quad \lJump F_u(h, u) \rJump = 0.
$$
If $\left.F_h(h, u)\right|_{x=x_L}^{x=x_R} = 0$ and $\left.F_u(h, u)\right|_{x=x_L}^{x=x_R} = 0$, then
  we have
  \begin{equation*}
\frac{d}{dt}\left(g,   h\right)_{\Omega} = 0, \quad \frac{d}{dt}\left(H, u\right)_{\Omega} = 0.
\end{equation*}
\end{theorem}
The theorem holds for Cauchy problems with compactly supported data and in a bounded domain with the periodic boundary conditions \eqref{eq:peridic_boundary_conditions}. If the numerical interface treatment satisfies a discrete analogue of Theorem \ref{theorem:conservation_principle}, we say that the numerical method is conservative.

\subsubsection{Energy conserving interface conditions} 
A second requirement of the interface conditions is that they should ensure energy stability for the coupled system. Hence, we need to determine the interface conditions such that the total energy in the domain is conserved. To this end, we introduce the following lemma.
\begin{lemma} \label{theoremenergy_interface} 
Consider the Serre equations \eqref{linearizedserre_minus}--\eqref{linearizedserre_plus}, and denote the energy in the subdomains $\Omega^{\pm}$ by $E^{\pm}(t)$. Considering only boundary contributions from the interface $x = 0$, we have the energy equation
\begin{equation*}
    \frac{d}{dt}(E^{-}(t) + E^{+}(t) ) = \left.\operatorname{IT}\right|_{x=0},
\end{equation*}
where 
$$
\operatorname{IT} = -\sum_{j=1}^5\lambda_j\left(({w}_j^{+})^2-({w}_j^{-})^2\right),
$$
and the vector $\mathbf{w}$ and the eigenvalues $\lambda_j$ are given by \eqref{eq:vector_w} and \eqref{eq:eigenvalues} respectively.
\end{lemma}
\begin{proof}
  Applying Lemma \ref{theoremenergy} to the equations \eqref{linearizedserre_minus} and \eqref{linearizedserre_plus} in $\Omega^{-}$ and $\Omega^{+}$ respectively gives the energy equations
  \begin{equation*} 
    \frac{dE^{-}}{dt} = \left.\operatorname{BT}^{-}\right|^{x=0}_{x=x_L} 
    \quad \text{and} \quad 
    \frac{dE^{+}}{dt} = \left.\operatorname{BT}^{+}\right|^{x=x_R}_{x=0}.
\end{equation*}
Here, the boundary terms $\operatorname{BT}^{\pm} = \operatorname{BT}(h^{\pm}, u^{\pm})$ are given by \eqref{boundarytermw}.
Summing the energy equations together and considering only boundary contributions from the interface $x =0$, we have
\begin{equation*} 
    \frac{d}{dt}(E^{-}(t) + E^{+}(t) ) = \left.\operatorname{BT}^{-}\right|_{x=0}-\left.\operatorname{BT}^{+}\right|_{x=0} =  -\left.\sum_{j=1}^5\lambda_j\left(({w}_j^{+})^2-({w}_j^{-})^2\right)\right|_{x =0}.
\end{equation*}
This completes the proof.
\end{proof}

Recall that we have four non-zero eigenvalues, namely $\lambda_2, \lambda_3, \lambda_4, \lambda_5$. It follows that we need four interface conditions to ensure well-posedness. The interface conditions should be imposed such that the interface terms vanish, that is $\left.\operatorname{IT}\right|_{x=0} = 0$, thus ensuring energy stability. 

Since $\lambda_1 = 0$, the interface terms can be written as
$$
\left.\operatorname{IT}\right|_{x=0} = \left. -\sum_{j=2}^5\lambda_j\left(\left({w}_j^{+}-{w}_j^{-}\right)\left({w}_j^{+}+{w}_j^{-}\right)\right)\right|_{x =0}.
$$
Hence, the interface conditions 
\begin{align}\label{eq:interface_conditions_1}
\lJump {w}_j \rJump =0, \quad j = 2, 3, 4, 5,
\end{align}
yields $\left.\operatorname{IT}\right|_{x=0} = 0$.
The interface conditions \eqref{eq:interface_conditions_1} can be equivalently rewritten as
\begin{align}\label{eq:interface_conditions}
\begin{split}
   &\lJump h \rJump =0, \quad \lJump u \rJump =0, \quad \lJump\frac{\partial u}{\partial x}\rJump = 0, \quad \lJump\frac{\partial^2 u}{\partial x^2}\rJump =0,
\end{split}
\end{align}
which also ensure that the interface term vanishes $\left.\operatorname{IT}\right|_{x=0} = 0$, and imply that the jump of the flux functions vanish, that is
$$
\lJump F_h(h, u) \rJump = U\lJump h\rJump + H\lJump u \rJump = 0, \quad \lJump F_u(h, u) \rJump = g \lJump h\rJump + U  \lJump u \rJump - \frac{H^2U}{3}\lJump\frac{\partial^2u}{\partial x^2}\rJump - \frac{H^2}{3}\frac{\partial}{\partial t}\lJump\frac{\partial u}{\partial x}\rJump = 0.
$$
The following theorem states that the interface conditions \eqref{eq:interface_conditions} ensure energy conservation.
\begin{theorem} \label{theoremenergy_interface_stability} 
Consider the Serre equations \eqref{linearizedserre_minus}--\eqref{linearizedserre_plus} subject to the interface conditions \eqref{eq:interface_conditions}, and let $E^{\pm}(t)$ denote the energy in the subdomains $\Omega^{\pm}$. Considering only boundary contributions from the interface $x = 0$, the total energy in the domain is conserved, that is
\begin{equation*} \label{energyequation_interface_stability}
    E^{-}(t) + E^{+}(t)  =  E^{-}(0) + E^{+}(0), \quad \forall t\ge 0.
\end{equation*}
\end{theorem}
\begin{proof}
Lemma \ref{theoremenergy_interface} gives the energy equation \begin{equation*} \label{energyequation_interface}
    \frac{d}{dt}(E^{-}(t) + E^{+}(t) ) = \left.\operatorname{IT}\right|_{x=0} = \left. -\sum_{j=2}^5\lambda_j\left(\left({w}_j^{+}-{w}_j^{-}\right)\left({w}_j^{+}+{w}_j^{-}\right)\right)\right|_{x =0},
\end{equation*}
where we have only considered boundary contributions from the interface $x = 0$. Since the interface conditions \eqref{eq:interface_conditions} imply \eqref{eq:interface_conditions_1}, the interface terms vanish $\left.\operatorname{IT}\right|_{x=0} = 0$, and hence
\begin{equation*} 
    \frac{d}{dt}(E^{-}(t) + E^{+}(t) ) = 0 \iff E^{-}(t) + E^{+}(t)  =  E^{-}(0) + E^{+}(0), \quad \forall t\ge 0.
\end{equation*}
The proof is complete.
\end{proof}

A stable and effective numerical method should as far as possible emulate the theoretical results established in Theorems \ref{theorem:well-posednes_ibvp_u0}, \ref{theorem:well-posednes_ibvp_UP} \ref{theorem:conservation_principle}, and \ref{theoremenergy_interface_stability}.

\section{Space discretization}

In this section, we present the DGSEM for the linearized Serre equations in a bounded domain. We will prove that the presented numerical method is conservative and stable by proving discrete analogues of Theorems \ref{theorem:well-posednes_ibvp_u0}, \ref{theorem:well-posednes_ibvp_UP} \ref{theorem:conservation_principle}, and \ref{theoremenergy_interface_stability}.

We begin by splitting the spatial domain $\Omega = [x_L,x_R]$ into $N$ uniform elements $I_k = [x_k, x_{k+1}]$ of length $\Delta x = x_{k+1}- x_{k}$, with $x_L = x_1 < x_2 < \ldots < x_{N+1} = x_R$. Each element $I_k$ can be mapped to a reference element $\widehat{\Omega} = [-1,1]$ using the following affine transformation:
\begin{equation} \label{transformation}
    \varphi_k(\xi) = x_{k} + \frac{\Delta x}{2}(\xi + 1), \quad \xi \in \widehat{\Omega} = [-1,1].
\end{equation}

Let $\mathcal{P}^P(\widehat{\Omega})$ be the space of polynomials of degree at most $P$ on $\widehat{\Omega}$. In spectral element methods, we use Lagrange polynomials 
\begin{equation*}
    \ell_j(\xi) = \prod_{\begin{smallmatrix}1\le i\le P+1\\ i\neq j\end{smallmatrix}} \frac{\xi-\xi_i}{\xi_j-\xi_i}, \quad j=1,2,\ldots,P+1,
\end{equation*}
as basis functions for the polynomial space $\mathcal{P}^P(\widehat{\Omega})$. Here, $\xi_1, \xi_2, \ldots, \xi_{P+1} \in \widehat{\Omega}$ are nodes of a Gaussian quadrature rule.
\subsection{Summation-by-parts spectral difference operators} \label{SBPsection}

We now derive discrete spatial operators that satisfy the SBP property in the reference element $\widehat{\Omega}$. We first observe that any function $u$ defined on $\widehat{\Omega}$ has the following Lagrange polynomial approximation $\widehat{u}$:
\begin{equation*} \label{polyapprox}
 \widehat{u}(\xi) = \sum_{j=1}^{P+1} u(\xi_j) \ell_j(\xi).
\end{equation*}
Hence, the weak derivative of $\widehat{u}$ is given by
\begin{equation} \label{weakderivative}
\left(\phi, \frac{d\widehat{u}}{d\xi} \right)_{\widehat{\Omega}} = \sum_{j=1}^{P+1} u(\xi_j) \left(\phi, \frac{d\ell_j}{d\xi} \right)_{\widehat{\Omega}}
\end{equation}
for all test function $\phi \in \mathcal{P}^P(\widehat{\Omega})$. Following the approach used in Galerkin spectral element methods, the test function is chosen such that $\phi = \ell_i$, $i = 1, 2, \ldots, P+1$, and we choose the Gauss-Lobatto quadrature rule. Since this quadrature rule is exact for polynomials of degree at most $2P - 1$, \eqref{weakderivative} can be written as
\begin{equation*}
\left(\ell_i, \frac{d\widehat{u}}{d\xi} \right)_{\widehat{\Omega}} = \sum_{j=1}^{P+1} u(\xi_j) \left(\ell_i, \frac{d\ell_j}{d\xi} \right)_{\widehat{\Omega}} = \sum_{j=1}^{P+1} Q_{ij} u(\xi_j),
\end{equation*}
where
\begin{equation} \label{Qmatrix}
Q_{ij} = \left(\ell_i, \frac{d\ell_j}{d\xi} \right)_{\widehat{\Omega}} = \sum_{k=1}^{P+1} \omega_k \ell_i(\xi_k) \left. \frac{d\ell_j}{d\xi} \right|_{\xi=\xi_k}.
\end{equation}
Here, $\omega_i > 0$ are the weights of the quadrature rule. Let $M = \operatorname{diag}(\omega_1, \omega_2, \ldots \allowbreak, \omega_{P+1})$ be the mass matrix, and define the discrete derivative operator $D = M^{-1}Q$. Using integration-by-parts, \eqref{Qmatrix} becomes
\begin{equation*}
Q_{ij} = -Q_{ij}^{\top} + \ell_i(1)\ell_j(1) - \ell_i(-1)\ell_j(-1),
\end{equation*}
and hence we obtain the SBP property for the first derivative operator $D$
\begin{equation} \label{sbp1}
(MD)_{ij} + (D^{\top}M)_{ij} = B_{ij},
\end{equation}
where $B = \operatorname{diag}(-1,0,0,\ldots,0,1)$. 

So far we have discussed how to derive spatial operators in the reference element $\widehat{\Omega}$. In a physical element $I_k$, the transformation \eqref{transformation} gives
\begin{equation} \label{discreteoperators}
    D_x = \frac{2}{\Delta x}D \quad \text{and} \quad M_x = \frac{\Delta x}{2} M.
\end{equation}
The operator $D_x$ is a discrete derivative operator that approximates the spatial derivative $\frac{\partial}{\partial x}$, and the mass matrix $M_x$ is an operator used to approximate integration. These operators also satisfy the SBP property
\begin{equation} \label{sbp2}
    M_x D_x + D^{\top}_x M_x = B,
\end{equation}
where $B$ is defined in \eqref{sbp1}.
We will approximate higher derivatives within a physical element using $D_x^{l} \approx {\partial^l }/{\partial x^l}$, for $l = 1, 2, 3$. 

The spatial derivative operators $D_x^{l}$ with $l = 1, 2, 3$ satisfy discrete analogues of the integration-by-parts principle \eqref{ibp_x}--\eqref{ibp_xxx}. To see this we introduce the following discrete inner product and norm
\begin{equation*}
 \langle \mathbf{u}, \mathbf{v} \rangle_{M_x}  = \mathbf{v}^{\top} M_x \mathbf{u}, \quad  \langle \mathbf{u}, \mathbf{u} \rangle_{M_x} = \| \mathbf{u} \|^2_{M_x} > 0,
\end{equation*}
for $\mathbf{u}, \mathbf{v} \in \mathbb{R}^{P+1}$. By straightforward computations using \eqref{sbp2}, we obtain the SBP properties for first, second and third derivatives as follows:
\begin{equation} \label{sbp_x}
      \left\langle \mathbf{u}, D_x \mathbf{v} \right\rangle_{M_x} = u_{P+1}v_{P+1} - u_1v_1 - \left\langle D_x \mathbf{u}, \mathbf{v} \right\rangle_{M_x},
\end{equation}
\begin{equation} \label{sbp_xx}
    \left\langle \mathbf{u}, D_x^2 \mathbf{v} \right\rangle_{M_x} = u_{P+1} \left(D_x \mathbf{v}\right)_{P+1} - u_1 \left(D_x \mathbf{v} \right)_1 - \left\langle D_x \mathbf{u}, D_x\mathbf{v} \right\rangle_{M_x},
\end{equation}
\begin{equation} \label{sbp_xxx}
\begin{split}
    \left\langle \mathbf{u}, D_x^3 \mathbf{v} \right\rangle_{M_x} &= u_{P+1} \left(D_x^2 \mathbf{v}\right)_{P+1} - u_1 \left(D_x^2 \mathbf{v} \right)_1 -\frac{1}{2}\left(\left(D_x \mathbf{u}\right)_{P+1} \left(D_x \mathbf{v}\right)_{P+1} - \left(D_x \mathbf{u}\right)_1 \left(D_x \mathbf{v} \right)_1\right)  \\
    &+\frac{1}{2}\left\langle D_x^2\mathbf{u}, D_x \mathbf{v} \right\rangle_{M_x}- \frac{1}{2}\left\langle D_x\mathbf{u}, D_x^2 \mathbf{v} \right\rangle_{M_x},
    \end{split}
\end{equation}
which are the discrete analogues of \eqref{ibp_x}--\eqref{ibp_xxx}. 

When $\mathbf{v} = \mathbf{u}$, \eqref{sbp_x} becomes
\begin{equation}
    \left\langle \mathbf{u}, D_x \mathbf{u} \right\rangle_{M_x} = \frac{1}{2}\left(u_{P+1}^2 - u_1^2\right) \label{sbpfirst}.
\end{equation}
For second and third order derivatives, substituting $\mathbf{v} = \dfrac{d\mathbf{u}}{dt}$ in \eqref{sbp_xx} and $\mathbf{v} = \mathbf{u}$ in \eqref{sbp_xxx} gives the following identities:
\begin{align}
    &\left\langle \mathbf{u}, D_x^2 \frac{d\mathbf{u}}{dt} \right\rangle_{M_x} = u_{P+1} \left(D_x \frac{d\mathbf{u}}{dt} \right)_{P+1} - u_1 \left(D_x \frac{d\mathbf{u}}{dt} \right)_1 - \left\langle D_x \mathbf{u}, D_x \frac{d\mathbf{u}}{dt} \right\rangle_{M_x} \label{sbpsecond}, \\
    &\left\langle \mathbf{u}, D_x^3 \mathbf{u} \right\rangle_{M_x} = u_{P+1} \left(D_x^2 \mathbf{u} \right)_{P+1} - u_1 \left(D_x^2 \mathbf{u} \right)_1 - \frac{1}{2}\left(D_x \mathbf{u}\right)^2_{P+1} + \frac{1}{2}\left(D_x \mathbf{u} \right)^2_1 \label{sbpthird}.
\end{align}
These identities \eqref{sbpsecond} and \eqref{sbpthird} correspond to the identities \eqref{secondorderidentity} and \eqref{thirdorderidentity} respectively.

 We conclude this section by stating a theorem regarding the accuracy of the discrete derivative operators $D_x^{l}$, which will be used later to derive error estimates.
\begin{theorem} \label{truncationerrortheorem}
Consider the discrete derivative operator $D_x$ and an element $I_k = [x_k, x_{k+1}]$ of length $\Delta x >0$. Let $x_k^{(j)} = x_k + \frac{\Delta x}{2}(\xi_j + 1)$, $j = 1,2,\ldots,P+1$ denote the Gauss-Lobatto quadrature nodes in $I_k$, with $\xi_j \in [-1,1]$, and  ${u}_j = u(x_k^{(j)})$ be the restriction of a sufficiently smooth function $u$ on the nodes. The truncation errors of the approximation of the partial derivatives ${\partial^l u}/{\partial x^l}$ are given by
\begin{align*}
&\left. D_x^{l}\mathbf{u} \right|_j = \left. \frac{\partial^l u}{\partial x^l} \right|_{x = x_k^{(j)}} + C_l \Delta x^{P+1-l} \left|\frac{\partial^{P+1}u}{\partial x^{P+1}}(\zeta_l) \right|, \quad l = 1, 2, 3,
\end{align*}
where $\zeta_l\in I_k$ and $C_l>0$ are constants independent of $\Delta x>0$.
\end{theorem}
\begin{proof}
The proof is a straightforward adaptation of the proof of Theorem 3 in \cite{howell1991derivative}.
\end{proof}
\subsection{Numerical interface treatments} \label{numericalinterfacetreatments}
Consider a two-element model $\Omega = \Omega^{-} \cup \Omega^{+}$, where $\Omega^{-}$ and $\Omega^{+}$ are subdomains defined as in \cref{sec:interfaceconditions}. We map each of $\Omega^{-}$ and $\Omega^{+}$ to the reference element $\widehat{\Omega} = [-1,1]$, and hence we have the discrete operators $D_x$ and $M_x$ as defined in $\eqref{discreteoperators}$ for each $\Omega^{-}$ and $\Omega^{+}$.

We write the nodal values of the variables as the following stacked vectors:
\begin{equation*}
\mathbf{u} = \left[\begin{matrix} \mathbf{u}^- \\ \mathbf{u}^+\end{matrix}\right],
\end{equation*}
where the minus and plus superscripts indicate the nodal values in $\Omega^{-}$ and $\Omega^{+}$ respectively. The discrete spatial operators are written as block-diagonal matrices
\begin{equation*} \label{blockoperators}
    \mathbf{D} = \begin{bmatrix}
D_x & \mathbf{0} \\
\mathbf{0} & D_x  \\
\end{bmatrix} \quad \text{and} \quad \mathbf{M} = \begin{bmatrix}
M_x & \mathbf{0} \\
\mathbf{0} & M_x  \\
\end{bmatrix},
\end{equation*}
and the mass matrix $\mathbf{M}$ gives the following discrete inner product and norm:
\begin{equation*}
    \langle \mathbf{u}, \mathbf{u} \rangle_{\mathbf{M}} = \mathbf{u}^{\top} \mathbf{M} \mathbf{u} = \| \mathbf{u} \|^2_{\mathbf{M}}.
\end{equation*}

By replacing the continuous spatial derivatives in \eqref{linearizedserre} with the discrete derivative operator $\mathbf{D}$, we obtain an element local semi-discrete numerical approximation
\begin{align} \label{semidiscrete}
\begin{split}
    &\frac{d\mathbf{h}}{dt} + \mathbf{D}(H\mathbf{u} + U\mathbf{h}) = \mathbf{0}, \\
    &\frac{d\mathbf{u}}{dt} + \mathbf{D}\left(g\mathbf{h} + U\mathbf{u} - \frac{H^2U}{3}\mathbf{D}^2\mathbf{u} - \frac{H^2}{3}\mathbf{D}\left( \frac{d\mathbf{u}}{dt}\right)\right) = \mathbf{0}.
\end{split}
\end{align}
The numerical approximation above has not imposed the interface conditions derived in \cref{sec:interfaceconditions}, and hence the numerical solutions in each $\Omega^{-}$ and $\Omega^{+}$ are still disconnected. The next challenge lies in connecting the solutions across the elements in an accurate and stable manner. In order to achieve this, we will impose the interface conditions using the SAT method \cite{carpenter1994time}.

We introduce the spatial interface operator
\begin{equation*}
 \widetilde{\mathbf{B}} = \begin{bmatrix}
\mathbf{e}_R\mathbf{e}_R^{\top} & -\mathbf{e}_R\mathbf{e}_L^{\top} \\
\mathbf{e}_L\mathbf{e}_R^{\top} & -\mathbf{e}_L\mathbf{e}_L^{\top}  \\
\end{bmatrix},
\end{equation*}
and the penalized derivative operator
\begin{equation} \label{penalizedoperator}
\widetilde{\mathbf{D}} = \mathbf{D} - \frac{1}{2} \mathbf{M}^{-1} \widetilde{\mathbf{B}},
\end{equation}
where $\mathbf{e}_R = [0, 0, \ldots, 0, 1]^{\top}$ and $\mathbf{e}_L = [1, 0, \ldots, 0, 0]^{\top}$. Hence, for any grid function $\mathbf{u} \in \mathbb{R}^{2P + 2}$, we have $\widetilde{\mathbf{B}}\mathbf{u} = [0, 0, \ldots, 0, {u}^-_{P+1} - {u}^+_1, {u}^-_{P+1} - {u}^+_1, 0, \ldots, 0, 0]^{\top}$. Note that for continuous functions we have ${u}^-_{P+1} - {u}^+_1 = 0$, it follows that $\widetilde{\mathbf{B}}\mathbf{u} = \mathbf{0}$, and hence $\widetilde{\mathbf{D}}\mathbf{u} = {\mathbf{D}}\mathbf{u}$. For later use we state the following lemma which is a discrete analogue to \eqref{ibp_x}, \eqref{secondorderidentity}, and \eqref{thirdorderidentity}.
\begin{lemma} \label{interfacesatlemma2}
Consider the penalized derivative operator $\widetilde{\mathbf{D}}$ defined in \eqref{penalizedoperator}. For all grid functions $\mathbf{u}, \mathbf{v} \in \mathbb{R}^{2P + 2}$ we have
\begin{align*}
    &\left\langle \mathbf{u}, \widetilde{\mathbf{D}} \mathbf{v} \right\rangle_{\mathbf{M}} +  \left\langle \widetilde{\mathbf{D}} \mathbf{u}, \mathbf{v} \right\rangle_{\mathbf{M}} = {u}^+_{P+1}{v}^+_{P+1} - {u}^-_1{v}^-_1, \\
    &\left\langle \mathbf{u}, \widetilde{\mathbf{D}}^2 \frac{d\mathbf{u}}{dt} \right\rangle_{\mathbf{M}} = u^+_{P+1}\left(D_x\frac{d\mathbf{u}^+}{dt}\right)_{P+1} -u^-_1\left(D_x\frac{d\mathbf{u}^-}{dt}\right)_1-\frac{1}{2} \frac{d}{dt} \|\widetilde{\mathbf{D}}\mathbf{u}\|_{\mathbf{M}}, \\
    &\left\langle \mathbf{u}, \widetilde{\mathbf{D}}^3 \mathbf{u} \right\rangle_{\mathbf{M}} = {u}^+_{P+1}\left(D_x^2\mathbf{u}^+\right)_{P+1} -{u}^-_1\left(D_x^2\mathbf{u}^-\right)_1-\frac{1}{2} \left(D_x \mathbf{u}^+ \right)^2_{P+1} + \frac{1}{2}\left(D_x \mathbf{u}^-\right)_1^2.
\end{align*}
\end{lemma}

To connect the solutions across the elements, we add interface penalty terms to the right hand side of \eqref{semidiscrete} as follows:
\begin{subequations} \label{interfacesat}
\begin{align}
    &\frac{d\mathbf{h}}{dt} + \mathbf{D}(H\mathbf{u} + U\mathbf{h}) = \tau_{11}H\mathbf{M}^{-1}\widetilde{\mathbf{B}}\mathbf{u} + \tau_{12}U\mathbf{M}^{-1}\widetilde{\mathbf{B}}\mathbf{h}  - \alpha_h  \mathbf{M}^{-1}\widetilde{\mathbf{B}}^{\top} \widetilde{\mathbf{B}}\mathbf{h}, \label{interfacesat1} \\
    &\frac{d\mathbf{u}}{dt} + \mathbf{D}\left(g\mathbf{h} + U\mathbf{u} - \frac{H^2U}{3}\mathbf{D}^2\mathbf{u} - \frac{H^2}{3}\mathbf{D}\left( \frac{d\mathbf{u}}{dt}\right)\right) = \tau_{21}g\mathbf{M}^{-1}\widetilde{\mathbf{B}}\mathbf{h} + \tau_{22}U\mathbf{M}^{-1}\widetilde{\mathbf{B}}\mathbf{u} \label{interfacesat2}\\
    &+\gamma_{21}H^2\mathbf{D}\mathbf{M}^{-1}\widetilde{\mathbf{B}}\frac{d\mathbf{u}}{dt} +\gamma_{22}H^2\mathbf{M}^{-1}\widetilde{\mathbf{B}}\mathbf{D}\frac{d\mathbf{u}}{dt}
    +\gamma_{23}H^2\mathbf{M}^{-1}\widetilde{\mathbf{B}}\mathbf{M}^{-1}\widetilde{\mathbf{B}}\frac{d\mathbf{u}}{dt} \nonumber \\
    &+\sigma_{21}H^2U \mathbf{D}^2\mathbf{M}^{-1}\widetilde{\mathbf{B}}\mathbf{u} + \sigma_{22}H^2U \mathbf{D}\mathbf{M}^{-1}\widetilde{\mathbf{B}} \mathbf{D}\mathbf{u} + \sigma_{23}H^2U \mathbf{D}\mathbf{M}^{-1}\widetilde{\mathbf{B}}\mathbf{M}^{-1}\widetilde{\mathbf{B}}\mathbf{u} \nonumber \\
    &+\sigma_{24}H^2U\mathbf{M}^{-1}\widetilde{\mathbf{B}}\mathbf{D}^2\mathbf{u} + \sigma_{25}H^2U\mathbf{M}^{-1}\widetilde{\mathbf{B}}\mathbf{D}\mathbf{M}^{-1}\widetilde{\mathbf{B}}\mathbf{u} + \sigma_{26}H^2U\mathbf{M}^{-1}\widetilde{\mathbf{B}}\mathbf{M}^{-1}\widetilde{\mathbf{B}}\mathbf{D}\mathbf{u} \nonumber \\
    &+\sigma_{27}H^2U\mathbf{M}^{-1}\widetilde{\mathbf{B}}\mathbf{M}^{-1}\widetilde{\mathbf{B}}\mathbf{M}^{-1}\widetilde{\mathbf{B}}\mathbf{u} - \alpha_u  \mathbf{M}^{-1}\widetilde{\mathbf{B}}^{\top} \widetilde{\mathbf{B}}\mathbf{u} \nonumber,
\end{align}
\end{subequations}
where $\tau_{ij}, \gamma_{ij}, \sigma_{ij}$ are real penalty parameters that will be chosen later to ensure stability, and $\alpha_h$, $\alpha_u \geq 0$ are real upwind parameters.
\begin{remark}
We note that the semi-discrete numerical approximation \eqref{interfacesat} is consistent with the interface conditions \eqref{eq:interface_conditions_1}. To see this, suppose that $h$ and $u$ are exact solutions of the Serre equations \eqref{linearizedserre_minus} and \eqref{linearizedserre_plus} subject to the interface conditions \eqref{eq:interface_conditions_1}. Since $h$ and $u$ are continuous at the interface, we have $\widetilde{\mathbf{B}}\mathbf{h} = \widetilde{\mathbf{B}}\mathbf{u} = \widetilde{\mathbf{B}}\frac{d\mathbf{u}}{dt} = \mathbf{0}$. Furthermore, the continuity of ${\partial u}/{\partial x}$ and ${\partial^2u}/{\partial x^2}$ at the interface implies $\widetilde{\mathbf{B}}\mathbf{D}\mathbf{u} = \widetilde{\mathbf{B}}\mathbf{D}\frac{d\mathbf{u}}{dt} = \widetilde{\mathbf{B}}\mathbf{D}^2\mathbf{u} \approx \mathbf{0}$, where we have ignored truncation errors that arise from spatial derivatives approximations. Therefore, any exact solutions that solves the Serre equations \eqref{linearizedserre_minus} and \eqref{linearizedserre_plus} subject to the interface conditions \eqref{eq:interface_conditions_1} will cause the penalty terms in the scheme \eqref{interfacesat} to vanish, and hence we recover \eqref{semidiscrete}.
\end{remark}

Let us now show that for a specific choice of the penalty parameters the numerical approximation \eqref{interfacesat} is conservative and stable. The following lemma states that under an appropriate choice of the penalty parameters, we can rewrite \eqref{interfacesat} into a more convenient form for our later analysis.
\begin{lemma} \label{interfacesatlemma}
Consider the semi-discrete numerical approximation \eqref{interfacesat} with the penalty parameters
\begin{align} \label{interfacepenalties}
\begin{split}
    &\tau_{11} = \tau_{12} = \tau_{21} = \tau_{22} = \frac{1}{2}, \quad \gamma_{21} = \gamma_{22} = \sigma_{21} = \sigma_{22} = \sigma_{24} = -\frac{1}{6}, \\
    &\gamma_{23} = \sigma_{23} = \sigma_{25} = \sigma_{26} = \frac{1}{12}, \quad \sigma_{27} = -\frac{1}{24}.
\end{split}
\end{align}
Then, \eqref{interfacesat} can be written as
\begin{subequations} \label{interfacepenalized}
\begin{align}
    &\frac{d\mathbf{h}}{dt} + \widetilde{\mathbf{D}}(H\mathbf{u} + U\mathbf{h}) + \alpha_h \mathbf{M}^{-1}\widetilde{\mathbf{B}}^{\top} \widetilde{\mathbf{B}}\mathbf{h}  = \mathbf{0}, \label{interfacepenalized1}\\
    &\frac{d\mathbf{u}}{dt} + \widetilde{\mathbf{D}}\left(g\mathbf{h} + U\mathbf{u} - \frac{H^2U}{3}\widetilde{\mathbf{D}}^2\mathbf{u} - \frac{H^2}{3}\widetilde{\mathbf{D}}\left( \frac{d\mathbf{u}}{dt}\right)\right) + \alpha_u  \mathbf{M}^{-1}\widetilde{\mathbf{B}}^{\top} \widetilde{\mathbf{B}}\mathbf{u} = \mathbf{0}, \label{interfacepenalized2}
\end{align}
\end{subequations}
where the operator $\widetilde{\mathbf{D}}$ is given by \eqref{penalizedoperator}.
\end{lemma}
The proof of the lemma involves mainly algebraic manipulations, and has been moved to \cref{interfacesatappendix}.

The theorems below state that the numerical approximation \eqref{interfacesat} is conservative and stable when the penalty parameters are chosen as specified in \eqref{interfacepenalties}.
\begin{theorem} \label{discreteconservationtheorem}
Consider the semi-discrete numerical approximation \eqref{interfacesat} with the penalty parameters \eqref{interfacepenalties} and upwind parameters that are real and positive, $\alpha_h \ge 0$ and $\alpha_u \ge 0$. Considering only boundary contributions from the interface, the numerical interface treatment  \eqref{interfacesat} is conservative, that is
\begin{equation*} \label{discreteconservation}
   g\left\langle \mathbf{1}, \frac{d\mathbf{h}}{dt}\right\rangle_{\mathbf{M}} = 0, \quad H\left\langle \mathbf{1}, \frac{d\mathbf{u}}{dt} \right\rangle_{\mathbf{M}} = 0,
\end{equation*}
where $\mathbf{1} = [1, 1, \cdots 1]\in \mathbb{R}^{2P+2}$.
\end{theorem}
\begin{proof}
Applying Lemma \ref{interfacesatlemma} and left multiplying \eqref{interfacepenalized1} by $g\mathbf{1}^{\top}\mathbf{M}$ and \eqref{interfacepenalized2} by $H\mathbf{1}^{\top}\mathbf{M}$, we have
\begin{align} \label{thm43eq}
\begin{split}
    &g\left\langle \mathbf{1}, \frac{d\mathbf{h}}{dt} \right\rangle_{\mathbf{M}} + g\left\langle \mathbf{1}, \widetilde{\mathbf{D}}(H\mathbf{u} + U\mathbf{h}) \right\rangle_{\mathbf{M}} +  \alpha_h g\left\langle \mathbf{1},\mathbf{M}^{-1}\widetilde{\mathbf{B}}^{\top} \widetilde{\mathbf{B}}\mathbf{h}\right\rangle_{\mathbf{M}}= 0, \\
    &H\left\langle \mathbf{1}, \frac{d\mathbf{u}}{dt} \right\rangle_{\mathbf{M}} + H\left\langle \mathbf{1}, \widetilde{\mathbf{D}}\left(g\mathbf{h} + U\mathbf{u} - \frac{H^2U}{3}\widetilde{\mathbf{D}}^2\mathbf{u} - \frac{H^2}{3}\widetilde{\mathbf{D}}\left( \frac{d\mathbf{u}}{dt}\right)\right) \right\rangle_{\mathbf{M}} +  \alpha_u H\left\langle \mathbf{1},\mathbf{M}^{-1}\widetilde{\mathbf{B}}^{\top} \widetilde{\mathbf{B}}\mathbf{u }\right\rangle_{\mathbf{M}} = 0.
\end{split}
\end{align}
Considering only boundary contributions from the interface, Lemma \ref{interfacesatlemma2} and the obvious equalities $\widetilde{\mathbf{D}}\mathbf{1} = \mathbf{0}$ and $\widetilde{\mathbf{B}}\mathbf{1} = \mathbf{0}$ imply that we have $\left\langle \mathbf{1}, \widetilde{\mathbf{D}}\mathbf{v} \right\rangle_{\mathbf{M}} = 0$ and $\left\langle \mathbf{1}, \mathbf{M}^{-1}\widetilde{\mathbf{B}}^{\top} \widetilde{\mathbf{B}}\mathbf{v}\right\rangle_{\mathbf{M}} = 0$, for all $\mathbf{v} \in \mathbb{R}^{2P+2}$. Hence, \eqref{thm43eq} becomes
\begin{equation*} \label{thm43eq2}
   g\left\langle \mathbf{1}, \frac{d\mathbf{h}}{dt}\right\rangle_{\mathbf{M}} = 0, \quad H\left\langle \mathbf{1}, \frac{d\mathbf{u}}{dt} \right\rangle_{\mathbf{M}} = 0.
\end{equation*}
This completes the proof.
\end{proof}
\begin{theorem} \label{discreteenergyconservationtheorem}
Consider the semi-discrete numerical approximation \eqref{interfacesat} with the penalty parameters \eqref{interfacepenalties}, and define the discrete energy
\begin{equation*} \label{discreteinterfaceenergy}
    E_{\mathbf{M}}(t) = \frac{g}{2}\| \mathbf{h} \|^2_{\mathbf{M}} + \frac{H}{2} \|\mathbf{u} \|^2_{\mathbf{M}} + \frac{H^3}{6} \| \widetilde{\mathbf{D}}\mathbf{u} \|^2_{\mathbf{M}},
\end{equation*}
where $\widetilde{\mathbf{D}}$ is defined in \eqref{penalizedoperator}. Considering only boundary contributions from the interface, the discrete energy is bounded by the discrete energy of the initial data, that is
\begin{equation}
   E_{\mathbf{M}}(t) \le E_{\mathbf{M}}(0), \quad \forall t\ge 0.
\end{equation}
\end{theorem}
\begin{proof}
We apply Lemma \ref{interfacesatlemma}, and then we left multiply \eqref{interfacepenalized1} and \eqref{interfacepenalized2} by $g\mathbf{h}^{\top}\mathbf{M}$ and $H\mathbf{u}^{\top}\mathbf{M}$ respectively. We have
\begin{align} \label{interfaceenergyproof1}
\begin{split}
    &g\left\langle \mathbf{h}, \frac{d \mathbf{h}}{dt} \right\rangle_{\mathbf{M}} + gH\left\langle \mathbf{h}, \widetilde{\mathbf{D}}\mathbf{u} \right\rangle_{\mathbf{M}} + gU\left\langle \mathbf{h}, \widetilde{\mathbf{D}}\mathbf{h} \right\rangle_{\mathbf{M}} + g\alpha_h\left({h}^-_{P+1} - {h}^+_1\right)^2  = 0, \\
    &H\left\langle \mathbf{u}, \frac{d \mathbf{u}}{dt} \right\rangle_{\mathbf{M}} + gH\left\langle \mathbf{u}, \widetilde{\mathbf{D}}\mathbf{h} \right\rangle_{\mathbf{M}} + HU\left\langle \mathbf{u}, \widetilde{\mathbf{D}}\mathbf{u} \right\rangle_{\mathbf{M}} \\
    &- \frac{H^3U}{3}\left\langle \mathbf{u}, \widetilde{\mathbf{D}}^3\mathbf{u} \right\rangle_{\mathbf{M}} - \frac{H^3}{3}\left\langle \mathbf{u}, \widetilde{\mathbf{D}}^2\frac{d\mathbf{u}}{dt} \right\rangle_{\mathbf{M}} + H\alpha_u\left({u}^-_{P+1} - {u}^+_1\right)^2 = 0.
\end{split}
\end{align}
Since we only consider boundary contributions from the interface,
Lemma \ref{interfacesatlemma2} gives
\begin{equation*}
    \left\langle \mathbf{h}, \widetilde{\mathbf{D}} \mathbf{h} \right\rangle_{\mathbf{M}} = \left\langle \mathbf{u}, \widetilde{\mathbf{D}} \mathbf{u} \right\rangle_{\mathbf{M}} = \left\langle \mathbf{u}, \widetilde{\mathbf{D}}^3 \mathbf{u} \right\rangle_{\mathbf{M}} = 0 \quad \text{and} \quad \left\langle \mathbf{u}, \widetilde{\mathbf{D}}^2 \frac{d\mathbf{u}}{dt} \right\rangle_{\mathbf{M}} = -\frac{1}{2}\frac{d}{dt}\| \widetilde{\mathbf{D}}\mathbf{u} \|^2_{\mathbf{M}},
\end{equation*}
and hence \eqref{interfaceenergyproof1} becomes
\begin{align*}
    &g\left\langle \mathbf{h}, \frac{d \mathbf{h}}{dt} \right\rangle_{\mathbf{M}} + gH\left\langle \mathbf{h}, \widetilde{\mathbf{D}}\mathbf{u} \right\rangle_{\mathbf{M}} +  g\alpha_h\left({h}^-_{P+1} - {h}^+_1\right)^2 = 0, \\
    &H\left\langle \mathbf{u}, \frac{d \mathbf{u}}{dt} \right\rangle_{\mathbf{M}} + gH\left\langle \mathbf{u}, \widetilde{\mathbf{D}}\mathbf{h} \right\rangle_{\mathbf{M}} + \frac{H^3}{6}\frac{d}{dt}\| \widetilde{\mathbf{D}}\mathbf{u} \|^2_{\mathbf{M}} +  H\alpha_u\left({u}^-_{P+1} - {u}^+_1\right)^2 = 0.
\end{align*}
Summing them together and using Lemma \ref{interfacesatlemma2} yields
\begin{equation}\label{eq:proof_end_0}
    \frac{d}{dt}\left(\frac{g}{2}\| \mathbf{h} \|^2_{\mathbf{M}} + \frac{H}{2} \|\mathbf{u} \|^2_{\mathbf{M}} + \frac{H^3}{6}\| \widetilde{\mathbf{D}}\mathbf{u} \|^2_{\mathbf{M}}\right) = - g\alpha_h\left({h}^-_{P+1} - {h}^+_1\right)^2 - H\alpha_u\left({u}^-_{P+1} - {u}^+_1\right)^2 \le 0,
\end{equation}
where we have used $\alpha_h \geq 0$ and $\alpha_u \geq 0$. We recognize that the left side of \eqref{eq:proof_end_0} is the time derivative of the discrete energy \eqref{discreteinterfaceenergy}, that is
\begin{equation}\label{eq:proof_end}
    \frac{d}{dt} E_{\mathbf{M}}(t) \le 0.
\end{equation}
Time integrating \eqref{eq:proof_end} completes the proof.
When the upwind parameters vanish, $\alpha_h =\alpha_u = 0$, the energy is conserved $E_{\mathbf{M}}(t) = E_{\mathbf{M}}(0)$.
\end{proof}

\subsection{Numerical boundary treatments}

In this section, we describe how to numerically enforce the boundary conditions derived in \cref{wellposedbcs} in a stable manner. To this end, we will utilize the SAT method to weakly impose the external boundary conditions. For simplicity, we will consider numerical approximations in a single element, and boundary contributions will be considered only one boundary at a time.

Let us begin by considering $U > 0$, which corresponds to \textbf{Case 2} in \cref{wellposedbcs}. From the continuous analysis in \cref{wellposedbcs}, we need to impose the boundary conditions \eqref{eq:bc_Up}. In order to ensure well-posedness, the constants $\alpha_j, \beta_j$ in \eqref{eq:bc_Up} have to chosen such that they satisfy the conditions of \cref{theorem:well-posednes_ibvp_UP}. 

For convenience, we only consider the following constants:
\begin{equation} \label{boundarycst1}
    \alpha_2 = \alpha_3 = \beta_2 = \beta_3 = 0 \quad \text{and} \quad \alpha_4 = \frac{1}{\beta_4} = \frac{C^-}{C^+},
\end{equation}
where $C^{\pm} = \sqrt{4 H^{4} + \left(3 U \pm \sqrt{4 H^{4} + 9 U^{2}}\right)^{2}}$. Thus, the boundary conditions \eqref{eq:bc_Up} become
\begin{subequations} \label{boundarynumerical1}
\begin{align}
    &h(x_L,t) = 0, \quad u(x_L,t) = 0, \quad \left. \frac{\partial u}{\partial x}\right|_{x=x_L} = 0, \label{boundarynumerical1_left}\\
    &u(x_R,t) = 0. \label{boundarynumerical1_right}
\end{align}
\end{subequations}
It is straightforward to check that the chosen constants \eqref{boundarycst1} satisfy the conditions of \cref{theorem:well-posednes_ibvp_UP}, and hence the boundary conditions \eqref{boundarynumerical1} are well-posed.

Let us now utilize the SAT method to impose the boundary conditions \eqref{boundarynumerical1} in a stable manner. As previously mentioned, we consider one boundary at a time in a single element. Starting from the left boundary $x = x_L$, a single element semi-discrete numerical approximation of the IBVP \eqref{linearizedserre}, \eqref{eq:initial_condition} and \eqref{boundarynumerical1_left} is given by
\begin{subequations} \label{boundarysat1}
\begin{align}
    &\frac{d\mathbf{h}}{dt} + D_x(H\mathbf{u} + U\mathbf{h}) = \tau_0 M_x^{-1} \mathbf{e}_L \mathbf{e}_L^{\top} (U \mathbf{h}) + \theta_0 M_x^{-1} \mathbf{e}_L \mathbf{e}_L^{\top} (H\mathbf{u}), \label{boundarysat1_1} \\
    &\frac{d\mathbf{u}}{dt} + D_x\left(g\mathbf{h} + U\mathbf{u} - \frac{H^2U}{3}D_x^2\mathbf{u} - \frac{H^2}{3}D_x\left( \frac{d\mathbf{u}}{dt}\right)\right) =
    \gamma_0 M_x^{-1}D_x^{\top} \mathbf{e}_L \mathbf{e}_L^{\top} \left( H^2 \frac{d\mathbf{u}}{dt}\right) \label{boundarysat1_2} \\
    &+\sigma_0 M_x^{-1} \mathbf{e}_L \mathbf{e}_L^{\top} M_x^{-1} \mathbf{e}_L \mathbf{e}_L^{\top} \left(H^2 \frac{d\mathbf{u}}{dt} \right) 
    + \eta_0  M_x^{-1}  \mathbf{e}_L \mathbf{e}_L^{\top} (U \mathbf{u}) 
    + \mu_0 M_x^{-1} D_x^{\top} \mathbf{e}_L \mathbf{e}_L^{\top} (H^2UD_x\mathbf{u}) \nonumber \\
    &+ \rho_0 M_x^{-1} (D_x^{\top})^2 \mathbf{e}_L \mathbf{e}_L^{\top} (H^2U\mathbf{u}), \nonumber
\end{align}
\end{subequations}
where $\tau_0, \theta_0, \gamma_0, \sigma_0, \eta_0, \mu_0, \rho_0$ are penalty parameters. The following theorem states that the numerical approximation \eqref{boundarysat1} is stable under a suitable choice of the penalty parameters.
\begin{theorem} \label{numboundarythm1}
Consider the numerical approximation \eqref{boundarysat1} with the penalty parameters $\tau_0 = \eta_0 = -\frac{1}{2}$, $\theta_0 = -1$, $\rho_0 = -\gamma_0 = -\sigma_0 = \frac{1}{3}$, $\mu_0 = -\frac{1}{6}$, and define the discrete energy
\begin{equation*} \label{boundaryenergy}
    E_{M_x}(t) = \frac{g}{2}\| \mathbf{h} \|^2_{M_x} + \frac{H}{2} \|\mathbf{u} \|^2_{M_x} + \frac{H^3}{6} \| \widetilde{D}_x\mathbf{u} \|^2_{M_x},
\end{equation*}
where $\widetilde{D}_x = D_x + M_x^{-1}\mathbf{e}_L\mathbf{e}_L^{\top}$. Considering only boundary contributions from the left boundary $x = x_L$, the discrete energy is conserved, that is
\begin{equation*}
    E_{M_x}(t) = E_{M_x}(0), \quad \forall t \ge 0.
\end{equation*}
\end{theorem}

\begin{proof}
Multiplying \eqref{boundarysat1_1} by $g\mathbf{h}^{\top}M_x$ and \eqref{boundarysat1_2} by $H\mathbf{u}^{\top}M_x$ yields
\begin{align}\label{boundarysat1proof1}
\begin{split}
&g\left \langle \mathbf{h}, \frac{d\mathbf{h}}{dt} \right \rangle_{M_x} + gH\left \langle \mathbf{h}, D_x \mathbf{u} \right \rangle_{M_x} + gU\left \langle \mathbf{h}, D_x \mathbf{h} \right \rangle_{M_x} = -\frac{gU}{2}h_1^2 -gH h_1 u_1, \\
&H\left \langle \mathbf{u}, \frac{d\mathbf{u}}{dt} \right \rangle_{M_x} + gH\left \langle \mathbf{u}, D_x \mathbf{h} \right \rangle_{M_x} + HU\left \langle \mathbf{u}, D_x \mathbf{u} \right \rangle_{M_x}  - \frac{H^3U}{3}\left \langle \mathbf{u}, D_x^3 \mathbf{u} \right \rangle_{M_x} - \frac{H^3}{3}\left \langle \mathbf{u}, D_x^2 \frac{d\mathbf{u}}{dt} \right \rangle_{M_x} \\ 
&= -\frac{H^3}{3}\mathbf{u}^{\top}D_x^{\top}\mathbf{e}_L \mathbf{e}_L^{\top} \frac{d\mathbf{u}}{dt} - \frac{H^3}{3} \mathbf{u}^{\top} \mathbf{e}_L \mathbf{e}_L^{\top} M_x^{-1} \mathbf{e}_L \mathbf{e}_L^{\top} \frac{d\mathbf{u}}{dt} - \frac{HU}{2}u_1^2 -\frac{H^3U}{6}\left(D_x\mathbf{u} \right)_1^2 + \frac{H^3U}{3}\left(D_x^2\mathbf{u}\right)_1u_1.
\end{split}
\end{align}

Applying the identities \eqref{sbpfirst}-\eqref{sbpthird} to \eqref{boundarysat1proof1} gives
\begin{subequations} \label{boundarysat1proof2}
\begin{align}
&g\left \langle \mathbf{h}, \frac{d\mathbf{h}}{dt} \right \rangle_{M_x} + gH\left \langle \mathbf{h}, D_x \mathbf{u} \right \rangle_{M_x} = -gH h_1 u_1 \label{boundarysat1proof2_1}, \\
&H\left \langle \mathbf{u}, \frac{d\mathbf{u}}{dt} \right \rangle_{M_x} + \frac{H^3}{3}\left \langle \left(D_x + M_x^{-1}\mathbf{e}_L\mathbf{e}_L^{\top} \right)\mathbf{u}, \left( D_x + M_x^{-1}\mathbf{e}_L\mathbf{e}_L^{\top}\right) \frac{d\mathbf{u}}{dt} \right \rangle_{M_x} + gH \left \langle \mathbf{u}, D_x \mathbf{h} \right \rangle_{M_x} = 0,\label{boundarysat1proof2_2}
\end{align}
\end{subequations}
where we have also used the obvious equations
\begin{equation*}
    u_1\left(D_x \frac{d\mathbf{u}}{dt} \right)_1 = \mathbf{u}^{\top} \mathbf{e}_L\mathbf{e}_L^{\top} D_x \frac{d\mathbf{u}}{dt} \quad \text{and} \quad \left \langle D_x\mathbf{u}, D_x\frac{d\mathbf{u}}{dt} \right \rangle_{M_x} = \mathbf{u}^{\top} D_x^{\top} M_x D_x \frac{d\mathbf{u}}{dt}.
\end{equation*}
Summing \eqref{boundarysat1proof2_1} and \eqref{boundarysat1proof2_2} together, and then using \eqref{sbp_x} gives 
\begin{equation*}
    \frac{d}{dt} \left(\frac{g}{2}\| \mathbf{h} \|^2_{M_x} + \frac{H}{2} \|\mathbf{u} \|^2_{M_x} + \frac{H^3}{6} \| \widetilde{D}_x\mathbf{u} \|^2_{M_x}\right) = 0.
\end{equation*}
where $\widetilde{D}_x = D_x + M_x^{-1}\mathbf{e}_L\mathbf{e}_L^{\top}$. Thus, we have 
\begin{equation*}
    \frac{d}{dt} E_{M_x}(t) = 0,
\end{equation*}
and we conclude the proof by time integrating this equation.
\end{proof}

For the right boundary $x = x_R$, we need to impose the boundary condition \eqref{boundarynumerical1_right}. A single element semi-discrete numerical approximation of the IBVP \eqref{linearizedserre}, \eqref{eq:initial_condition} and \eqref{boundarynumerical1_right} is as follows:
\begin{subequations} \label{boundarysat2}
\begin{align}
    &\frac{d\mathbf{h}}{dt} + D_x(H\mathbf{u} + U\mathbf{h}) = \theta_N M_x^{-1} \mathbf{e}_R \mathbf{e}_R^{\top} (H\mathbf{u}), \label{boundarysat2_1} \\
    &\frac{d\mathbf{u}}{dt} + D_x\left(g\mathbf{h} + U\mathbf{u} - \frac{H^2U}{3}D_x^2\mathbf{u} - \frac{H^2}{3}D_x\left( \frac{d\mathbf{u}}{dt}\right)\right) =
    \gamma_N M_x^{-1}D_x^{\top} \mathbf{e}_R \mathbf{e}_R^{\top} \left( H^2 \frac{d\mathbf{u}}{dt}\right) \label{boundarysat2_2} \\
    &+\sigma_N M_x^{-1} \mathbf{e}_R \mathbf{e}_R^{\top} M_x^{-1} \mathbf{e}_R \mathbf{e}_R^{\top} \left(H^2 \frac{d\mathbf{u}}{dt} \right) 
    + \rho_N M_x^{-1} (D_x^{\top})^2 \mathbf{e}_R \mathbf{e}_R^{\top} (H^2U\mathbf{u}), \nonumber
\end{align}
\end{subequations}
where $\theta_N, \gamma_N, \sigma_N, \rho_N$ are penalty parameters. The theorem below states that the numerical approximation \eqref{boundarysat2} is stable under a specific choice of the penalty parameters.
\begin{theorem} \label{numboundarythm2}
Consider the numerical approximation \eqref{boundarysat2} with the penalty parameters $\theta_N = 1$, $\gamma_N = -\sigma_N = -\rho_N = \frac{1}{3}$, and define the discrete energy
\begin{equation*} \label{boundaryenergy2}
    E_{M_x}(t) = \frac{g}{2}\| \mathbf{h} \|^2_{M_x} + \frac{H}{2} \|\mathbf{u} \|^2_{M_x} + \frac{H^3}{6} \| \widetilde{D}_x\mathbf{u} \|^2_{M_x},
\end{equation*}
where $\widetilde{D}_x = D_x - M_x^{-1}\mathbf{e}_R\mathbf{e}_R^{\top}$. Considering only boundary conditions from the right boundary $x = x_R$, the discrete energy is bounded by the discrete energy of the initial data, that is
\begin{equation*}
    E_{M_x}(t) \leq E_{M_x}(0), \quad \forall t \ge 0.
\end{equation*}
\end{theorem}
\begin{proof}
Similar to the previous case, we multiply \eqref{boundarysat2_1} by $g\mathbf{h}^{\top}M_x$ and \eqref{boundarysat2_2} by $H\mathbf{u}^{\top}M_x$. We have
\begin{align} \label{boundarysat2proof1}
\begin{split}
&g\left \langle \mathbf{h}, \frac{d\mathbf{h}}{dt} \right \rangle_{M_x} + gH\left \langle \mathbf{h}, D_x \mathbf{u} \right \rangle_{M_x} + gU\left \langle \mathbf{h}, D_x \mathbf{h} \right \rangle_{M_x} = gHh_{p+1}u_{p+1}, \\
&H\left \langle \mathbf{u}, \frac{d\mathbf{u}}{dt} \right \rangle_{M_x} + gH\left \langle \mathbf{u}, D_x \mathbf{h} \right \rangle_{M_x} + HU\left \langle \mathbf{u}, D_x \mathbf{u} \right \rangle_{M_x}  - \frac{H^3U}{3}\left \langle \mathbf{u}, D_x^3 \mathbf{u} \right \rangle_{M_x} - \frac{H^3}{3}\left \langle \mathbf{u}, D_x^2 \frac{d\mathbf{u}}{dt} \right \rangle_{M_x} \\ 
&= \frac{H^3}{3}\mathbf{u}^{\top}D_x^{\top}\mathbf{e}_R \mathbf{e}_R^{\top} \frac{d\mathbf{u}}{dt} - \frac{H^3}{3} \mathbf{u}^{\top} \mathbf{e}_R \mathbf{e}_R^{\top} M_x^{-1} \mathbf{e}_R \mathbf{e}_R^{\top} \frac{d\mathbf{u}}{dt} -\frac{H^3U}{3}(D^2_x \mathbf{u})_{p+1}u_{p+1}.
\end{split}
\end{align}
The identities \eqref{sbpfirst}-\eqref{sbpthird} imply that \eqref{boundarysat2proof1} becomes
\begin{align}\label{boundarysat2proof2}
\begin{split}
&g\left \langle \mathbf{h}, \frac{d\mathbf{h}}{dt} \right \rangle_{M_x} + gH\left \langle \mathbf{h}, D_x \mathbf{u} \right \rangle_{M_x} = gHh_{p+1}u_{p+1} - \frac{gU}{2}h^2_{p+1}, \\
&H\left \langle \mathbf{u}, \frac{d\mathbf{u}}{dt} \right \rangle_{M_x} + \frac{H^3}{3}\left \langle \left(D_x - M_x^{-1}\mathbf{e}_R\mathbf{e}_R^{\top} \right)\mathbf{u}, \left( D_x - M_x^{-1}\mathbf{e}_R\mathbf{e}_R^{\top}\right) \frac{d\mathbf{u}}{dt} \right \rangle_{M_x} + gH \left \langle \mathbf{u}, D_x \mathbf{h} \right \rangle_{M_x} \\ &= 
-\frac{gU}{2}u^2_{p+1} - \frac{H^3U}{6}(D_x \mathbf{u})^2_{p+1}.
\end{split}
\end{align}
Adding the equations in \eqref{boundarysat2proof2}, and then using \eqref{sbp_x} yields
\begin{equation*}
    \frac{d}{dt} \left(\frac{g}{2}\| \mathbf{h} \|^2_{M_x} + \frac{H}{2} \|\mathbf{u} \|^2_{M_x} + \frac{H^3}{6} \| \widetilde{D}_x\mathbf{u} \|^2_{M_x}\right) = -\frac{gU}{2}\left(h^2_{p+1} + u^2_{p+1}\right) - \frac{H^3U}{6}(D_x\mathbf{u})^2_{p+1} \le 0,
\end{equation*}
where $\widetilde{D}_x = D_x - M_x^{-1}\mathbf{e}_R\mathbf{e}_R^{\top}$. Time integrating this inequality completes the proof.
\end{proof}

We recall that the case $U = 0$ corresponds to \textbf{Case 1} in \cref{wellposedbcs}. In this case, choosing the constants $\alpha = \beta = 1$ in \eqref{bczero} gives the boundary conditions
\begin{equation*}
u(x_L,t) = 0 \quad \text{and} \quad u(x_R,t) = 0.
\end{equation*}
These boundary conditions can be immediately treated by substituting $U = 0$ in the numerical approximations \eqref{boundarysat1} and \eqref{boundarysat2} respectively.

\section{Error estimates}
In this section, we derive error estimates for the semi-discrete numerical approximation in the energy norm. For simplicity, the error analysis will be carried out in the two-element model $\Omega = \Omega^{-} \cup \Omega^{+}$ as in \cref{numericalinterfacetreatments}, and we only consider the boundary conditions \eqref{boundarynumerical1}. We note that the analysis can be easily generalized to multiple elements and other boundary conditions. 

Let the vectors $\mathbf{h}_{\text{ex}}$, $\mathbf{u}_{\text{ex}}$ denote the exact solution evaluated at the quadrature nodes. The pointwise error vectors can be written as
\begin{equation*}
\mathbf{E}_{\mathbf{h}} = \mathbf{h} - \mathbf{h}_{\text{ex}}, \quad \mathbf{E}_{\mathbf{u}} = \mathbf{u} - \mathbf{u}_{\text{ex}}.
\end{equation*}
Utilizing \eqref{interfacepenalized}, we obtain the error equations
\begin{subequations}
\begin{align} \label{errorequations1}
\frac{d\mathbf{E}_{\mathbf{h}}}{dt} + \widetilde{\mathbf{D}}(H\mathbf{E}_{\mathbf{u}} + U\mathbf{E}_{\mathbf{h}} ) + \alpha_h \mathbf{M}^{-1}\widetilde{\mathbf{B}}^{\top} \widetilde{\mathbf{B}}\mathbf{E}_\mathbf{h}  \ &= \mathbf{P}_b^{(1)} + \mathbf{T}^{(1)}, \\
\label{errorequations2}
\frac{d\mathbf{E}_{\mathbf{u}}}{dt} + \widetilde{\mathbf{D}}\left(g\mathbf{E}_{\mathbf{h}} + U\mathbf{E}_{\mathbf{u}} - \frac{H^2U}{3}\widetilde{\mathbf{D}}^2\mathbf{E}_{\mathbf{u}} - \frac{H^2}{3}\widetilde{\mathbf{D}}\left( \frac{d\mathbf{E}_{\mathbf{u}}}{dt}\right)\right) + \alpha_u \mathbf{M}^{-1}\widetilde{\mathbf{B}}^{\top} \widetilde{\mathbf{B}}\mathbf{E}_\mathbf{u}  &= \mathbf{P}_b^{(2)} + \mathbf{T}^{(2)},
\end{align}
\end{subequations}
where $\mathbf{T}^{(1)}$, $\mathbf{T}^{(2)}$ are quadrature nodes truncation errors and $\mathbf{P}^{(1)}_b$, $\mathbf{P}^{(2)}_b$ are boundary penalty terms. Theorems \ref{numboundarythm1} and \ref{numboundarythm2} imply that we have
\begin{align*}
&\mathbf{P}^{(1)}_b = \begin{bmatrix}
-\frac{1}{2} M_x^{-1} \mathbf{e}_L \mathbf{e}_L^{\top} \left(U \mathbf{E}_\mathbf{h}^-\right) - M_x^{-1} \mathbf{e}_L \mathbf{e}_L^{\top} \left(H \mathbf{E}_\mathbf{u}^- \right) \\ \\ M_x^{-1} \mathbf{e}_R \mathbf{e}_R^{\top}\left(H\mathbf{E}_\mathbf{u}^+ \right)
\end{bmatrix}, \\
&\mathbf{P}^{(2)}_b = \begin{bmatrix}
 -\frac{1}{3}M_x^{-1}D^{\top}_x \mathbf{e}_L \mathbf{e}_L^{\top}\left(H^2 \frac{d}{dt} \mathbf{E}_\mathbf{u}^- \right) - \frac{1}{3}M_x^{-1}\mathbf{e}_L \mathbf{e}_L^{\top}M_x^{-1}\mathbf{e}_L \mathbf{e}_L^{\top}\left(H^2 \frac{d}{dt} \mathbf{E}_\mathbf{u}^- \right) - \frac{1}{2}M_x^{-1}\mathbf{e}_L \mathbf{e}_L^{\top}\left(U \mathbf{E}_\mathbf{u}^- \right) \\
-\frac{1}{6}M_x^{-1}D^{\top}_x\mathbf{e}_L \mathbf{e}_L^{\top}\left( H^2UD_x \mathbf{E}_\mathbf{u}^-\right) + \frac{1}{3}M_x^{-1}\left(D^{\top}_x\right)^2\mathbf{e}_L \mathbf{e}_L^{\top}\left(H^2U \mathbf{E}_\mathbf{u}^-\right)  \\ \\ \frac{1}{3}M_x^{-1}D_x^{\top} \mathbf{e}_R\mathbf{e}_R^{\top}\left(H^2 \frac{d}{dt} \mathbf{E}_\mathbf{u}^+\right) 
-\frac{1}{3}M_x^{-1}\mathbf{e}_R\mathbf{e}_R^{\top}M_x^{-1}\mathbf{e}_R\mathbf{e}_R^{\top}\left(H^2 \frac{d}{dt}\mathbf{E}_\mathbf{u}^+\right) - \frac{1}{3}M_x^{-1}\left(D_x^{\top}\right)^2 \mathbf{e}_R\mathbf{e}_R^{\top}\left(H^2U\mathbf{E}_\mathbf{u}^+\right)
\end{bmatrix}.
\end{align*}

Let us now use Theorem \ref{truncationerrortheorem} to determine the order of the truncation errors $\mathbf{T}^{(1)}$ and $\mathbf{T}^{(2)}$. In the left hand side of \eqref{errorequations1}, the operator $\widetilde{\mathbf{D}}$ leads to a $O(\Delta x^{P})$ truncation error. The penalty term $\mathbf{P}^{(1)}_b$ in the right hand side of this equation does not contribute to the truncation error, and hence $\mathbf{T}^{(1)}$ is $O(\Delta x^{P})$. For \eqref{errorequations2}, the second term in the left hand side contains the operator $\widetilde{\mathbf{D}}^3$, which leads to a $O(\Delta x^{P-2})$ truncation error. In the right hand side of \eqref{errorequations2}, the only term in $\mathbf{P}^{(2)}_b$ that contributes to the truncation error is $-\frac{1}{6}M_x^{-1}D^{\top}_x\mathbf{e}_L \mathbf{e}_L^{\top}\left( H^2UD_x \mathbf{E}_\mathbf{u}^-\right)$. The operator $D_x$ in this term leads to a $O(\Delta x^{P})$ truncation error. Since both of the operators $M_x^{-1}$ and $D_x^{\top}$ have the factor $\Delta x^{-1}$, combining them together with the truncation error from $D_x$ gives a $O(\Delta x^{P-2})$ truncation. Hence, $\mathbf{T}^{(2)}$ is $O(\Delta x^{P-2})$.

We define the energy norm for the error
\begin{equation*}
\mathcal{E} = \left(\frac{g}{2} \left\|\mathbf{E}_{\mathbf{h}} \right\|^2 + \frac{H}{2} \left\|\mathbf{E}_{\mathbf{u}} \right\|^2 + \frac{H^3}{6}\left\|\widehat{\mathbf{D}}\mathbf{E}_{\mathbf{u}} \right\|^2\right)^{\frac{1}{2}},
\end{equation*}
where
\begin{equation*}
\widehat{\mathbf{D}} = \widetilde{\mathbf{D}} + \mathbf{M}^{-1} \begin{bmatrix}
\mathbf{\mathbf{e}_L\mathbf{e}_L^{\top}} & \mathbf{0} \\
\mathbf{0} & -\mathbf{e}_R\mathbf{e}_R^{\top}
\end{bmatrix}.
\end{equation*}
The following theorem gives a bound on the energy norm error.
\begin{theorem} 
For the semi-discrete numerical approximation \eqref{interfacepenalized} with the boundary treatments \eqref{boundarysat1} and \eqref{boundarysat2}, the energy norm error satisfies
\begin{equation*}
\mathcal{E} \leq \widetilde{C} \Delta x^{P-2} \int_{0}^{T} \left(\left|\frac{\partial^{P+1}h}{\partial x^{P+1}} \right|_{\infty} + \left|\frac{\partial^{P+1}u}{\partial x^{P+1}} \right|_{\infty}\right) \: dt,
\end{equation*}
where $\widetilde{C} > 0$ is independent of $\Delta x$.
\end{theorem}
\begin{proof}
The proof is omitted since it closely follows the proofs of Theorems \ref{numboundarythm1} and \ref{numboundarythm2}.
\end{proof}
Therefore, the error in the  energy norm converges to zero at the rate $O(\Delta {x}^{P-2})$.

\section{Numerical experiments} \label{numericalexperiments}

In this section, numerical experiments in 1D are presented to verify the theoretical results in this paper. Specifically, we first verify that our proposed numerical method is conservative, and then we present numerical convergence tests to demonstrate the accuracy of the numerical method. Finally, we perform numerical tests to demonstrate the advantage of high order accuracy in resolving highly oscillatory dispersive waves.

Let us now briefly describe the test problem that we will use for conservation and convergence tests. The analytical solution of the Cauchy problem for the Serre equations \eqref{linearizedserre} with the initial conditions 
\begin{equation*} \label{initcondtest}
h(x,0) = H + \frac{1}{\omega}\left(1 + \sin (\omega x)\right) \quad \text{and} \quad u(x,0) = U + \frac{c-U}{\omega H}\sin(\omega x)
\end{equation*}
is given by 
\begin{align} \label{analyticalsolution}
\begin{split}
    h(x,t) &= H + \frac{1}{\omega}\left(1 + \sin(\omega(x - ct))\right), \\
    u(x,t) &= U + \frac{c-U}{\omega H}\sin(\omega(x - ct)),
\end{split}
\end{align}
where $\omega = \frac{\sqrt{3(gH - (c - U)^2)}}{(c - U)H}$, and $c$ is chosen such that $U < c < U + \sqrt{gH}$. If the chosen domain is $\Omega = \left[x_L, x_L + \frac{2\pi n}{\omega}\right]$ for some $x_L \in \mathbb{R}$ and $n \in \{1,2,\ldots\}$, the analytical solution \eqref{analyticalsolution} satisfies the periodic boundary conditions \eqref{eq:peridic_boundary_conditions}, and hence we have a periodic boundary conditions problem. Otherwise, we have an IBVP with nonhomogeneous boundary conditions, where the boundary data are generated by the analytical solution \eqref{analyticalsolution}.

We use the model parameters $g = 9.8$, $H = 1.0$, $c = 0.5$, and choose two different background flow velocities $U = 0$ and $U = 0.2$. The domain $\Omega$ is chosen to be either $[0,\frac{2\pi}{\omega}]$ or $[0,1]$. We note that the domain $[0,\frac{2\pi}{\omega}]$ leads to a periodic boundary conditions problem, and $[0,1]$ leads to an IBVP with nonhomogeneous boundary conditions.

\subsection{Conservation test} \label{conservationtest}

We consider the periodic boundary conditions problem described in the previous section. The domain is discretized into $N = 20$ uniform elements, and we use polynomials of degree $P = 4$ to compute the numerical approximation. We consider two different choices of upwind parameters $\alpha_h = \alpha_u = 0$ and $\alpha_h = \alpha_u = 1$, and the numerical approximation is evolved using the classical explicit fourth order accurate Runge-Kutta method with the fixed time step $\Delta t = 10^{-3}$ until the final time $T = 1.0$.

To verify the discrete conservation properties of our numerical approximation, at each time step we compute the differences $g\Delta h=g\left\langle \mathbf{1},\mathbf{h}(t + \Delta t) \right\rangle_{\mathbf{M}} - g\left\langle \mathbf{1},\mathbf{h}(t) \right\rangle_{\mathbf{M}}$, $H\Delta u =H\left\langle \mathbf{1},\mathbf{u}(t + \Delta t) \right\rangle_{\mathbf{M}} - H\left\langle \mathbf{1},\mathbf{u}(t) \right\rangle_{\mathbf{M}}$, and $\Delta E=E_{\mathbf{M}}(t + \Delta t) - E_{\mathbf{M}}(t)$ using two different upwind parameters, $\alpha_h = \alpha_u = 0$ and $\alpha_h = \alpha_u = 1$, where $E_{\mathbf{M}}$ is given by \eqref{discreteinterfaceenergy}. The results are shown in Fig. \ref{conservationfig1}-\ref{conservationfig2}. We observe that when $\alpha_h = \alpha_u = 0$, all considered quantities are conserved up to machine precision. When $\alpha_h = \alpha_u = 1$, $E_{\mathbf{M}}$ is always negative, which implies that it is slightly decreasing. These observed results are consistent with Theorems \ref{discreteconservationtheorem} and \ref{discreteenergyconservationtheorem}.

\begin{figure}[H]
     \centering
     \begin{subfigure}[b]{0.325\textwidth}
         \includegraphics[width=\textwidth]{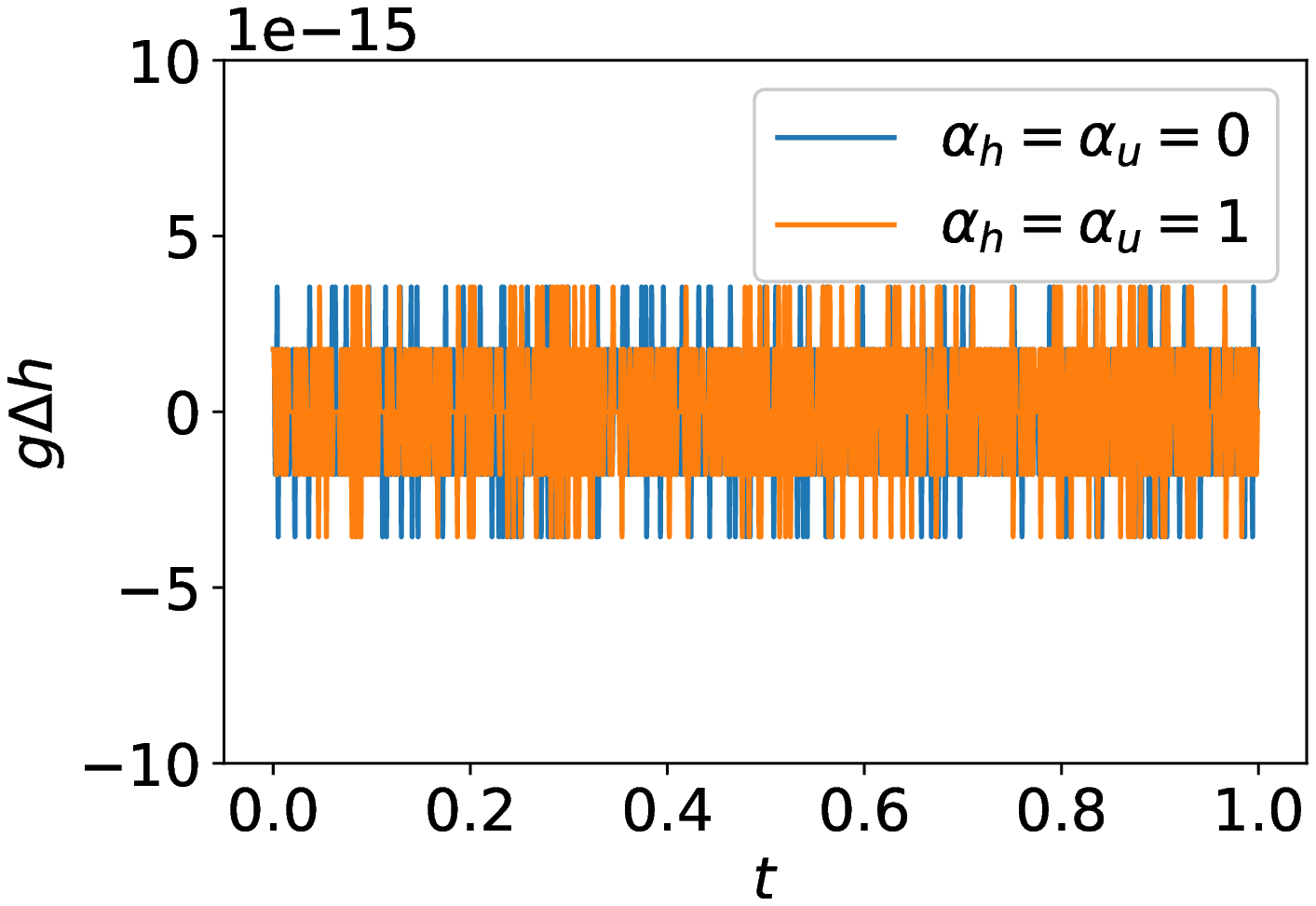}
         \caption{}
     \end{subfigure}
    \begin{subfigure}[b]{0.325\textwidth}
         \includegraphics[width=\textwidth]{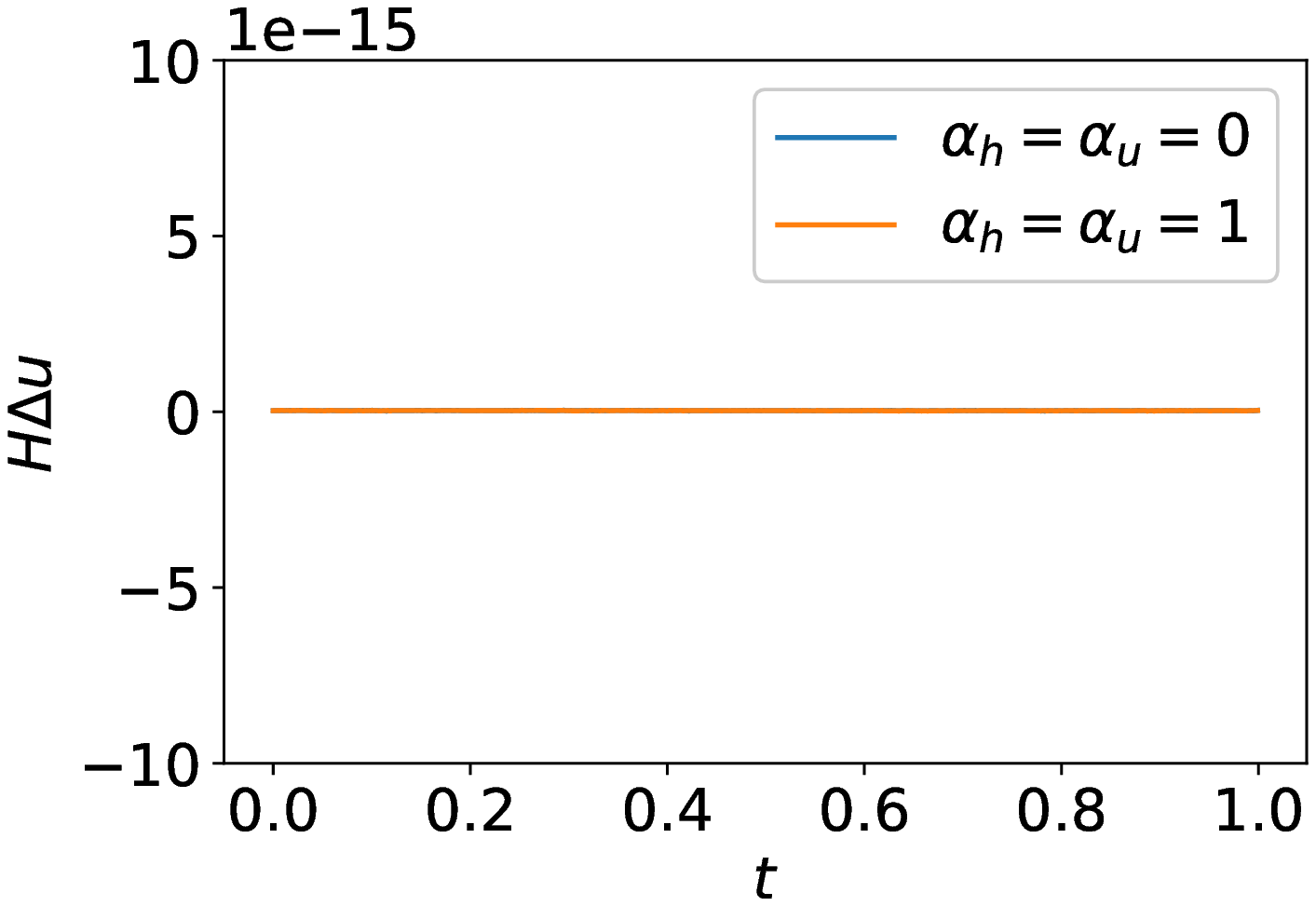}
         \caption{}
     \end{subfigure}
     \begin{subfigure}[b]{0.325\textwidth}
         \includegraphics[width=\textwidth]{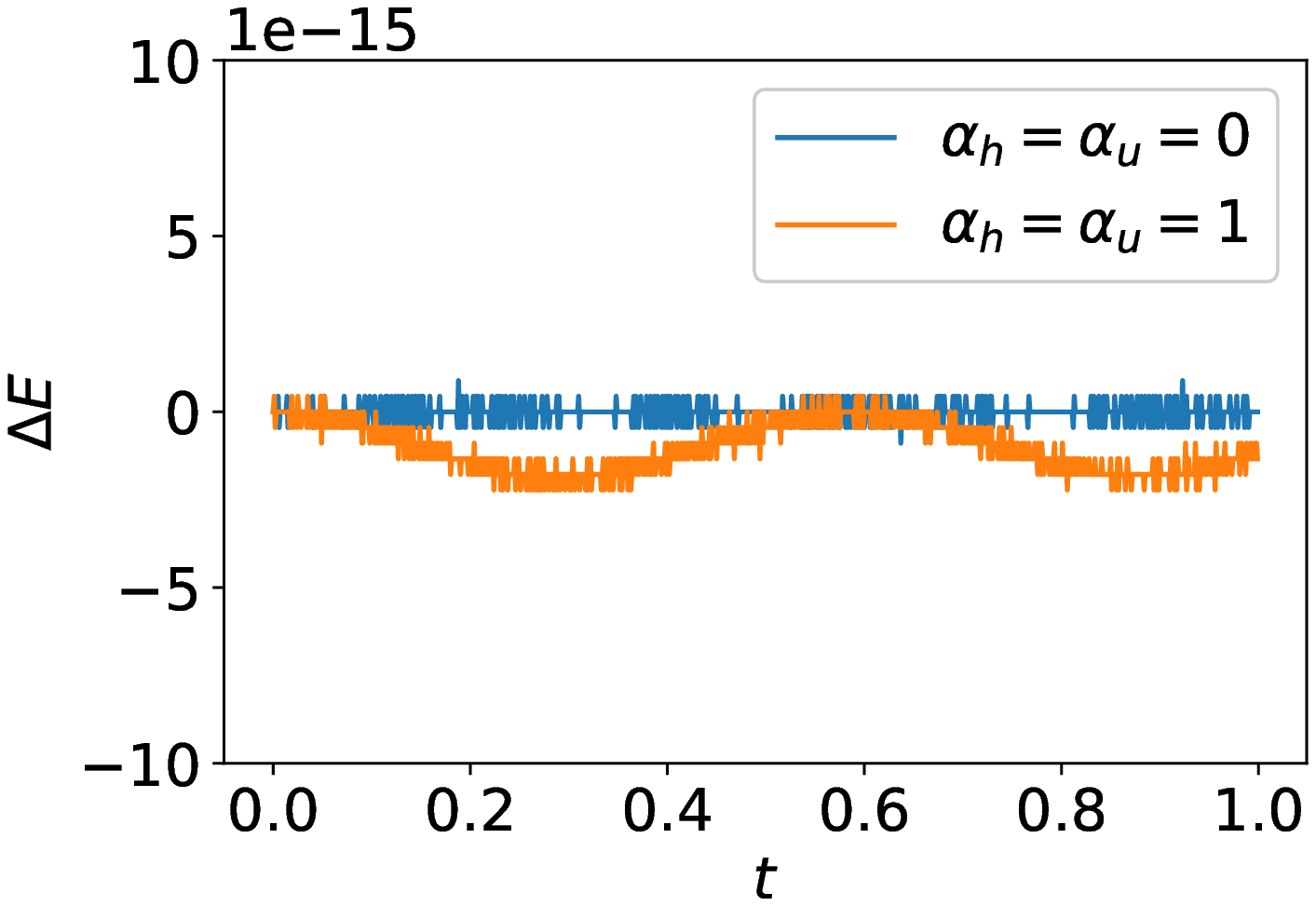}
         \caption{}
     \end{subfigure}
     \caption{Conservation test for (a) mass, (b) momentum and (c) energy achieving machine precision tolerance for the periodic boundary conditions problem with $U = 0$ for different upwind parameters $\alpha_h$ and $\alpha_u$.} \label{conservationfig1}
\end{figure}

\begin{figure}[H]
     \centering
     \begin{subfigure}[b]{0.325\textwidth}
         \includegraphics[width=\textwidth]{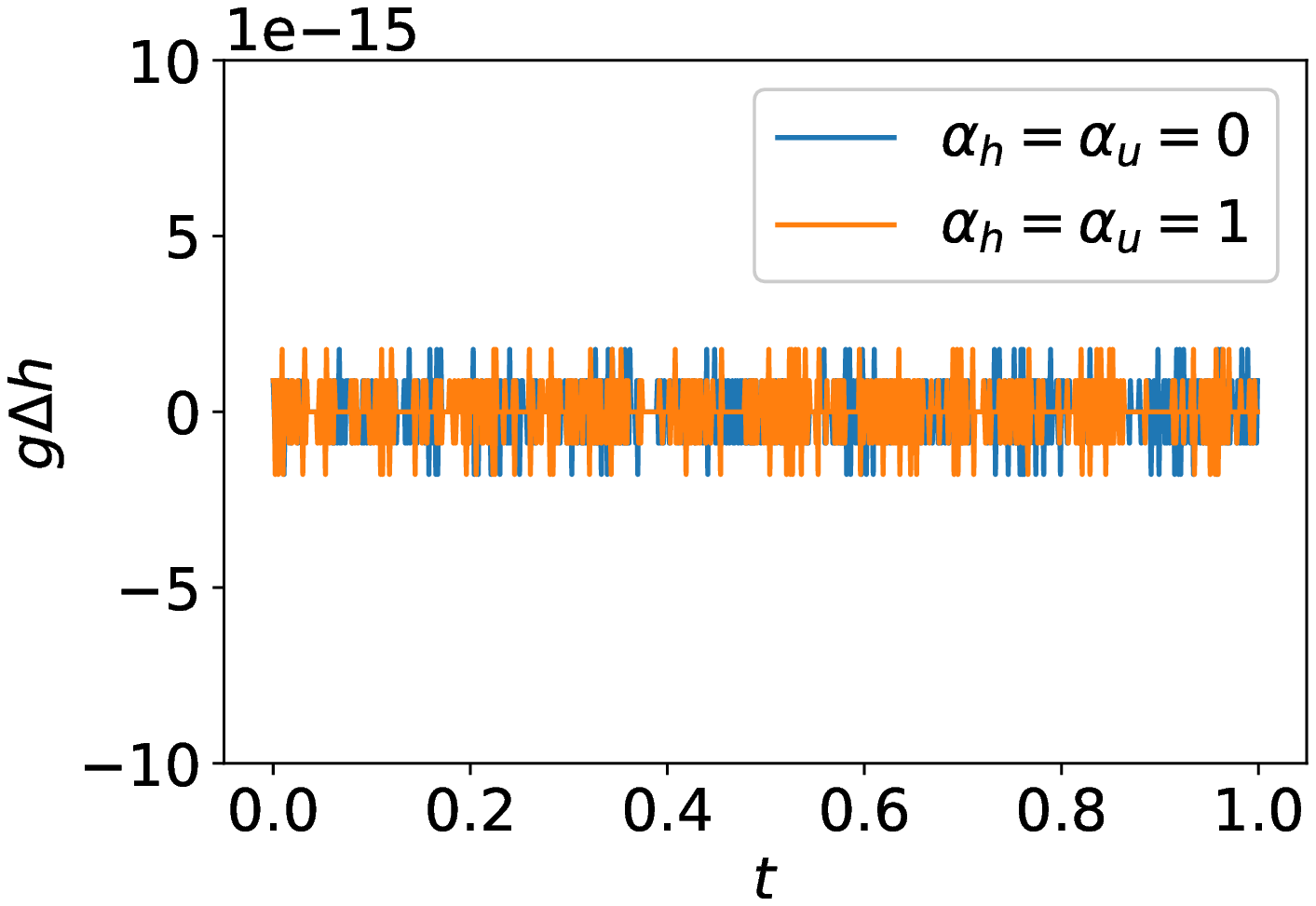}
          \caption{}
     \end{subfigure}
     \begin{subfigure}[b]{0.325\textwidth}
         \includegraphics[width=\textwidth]{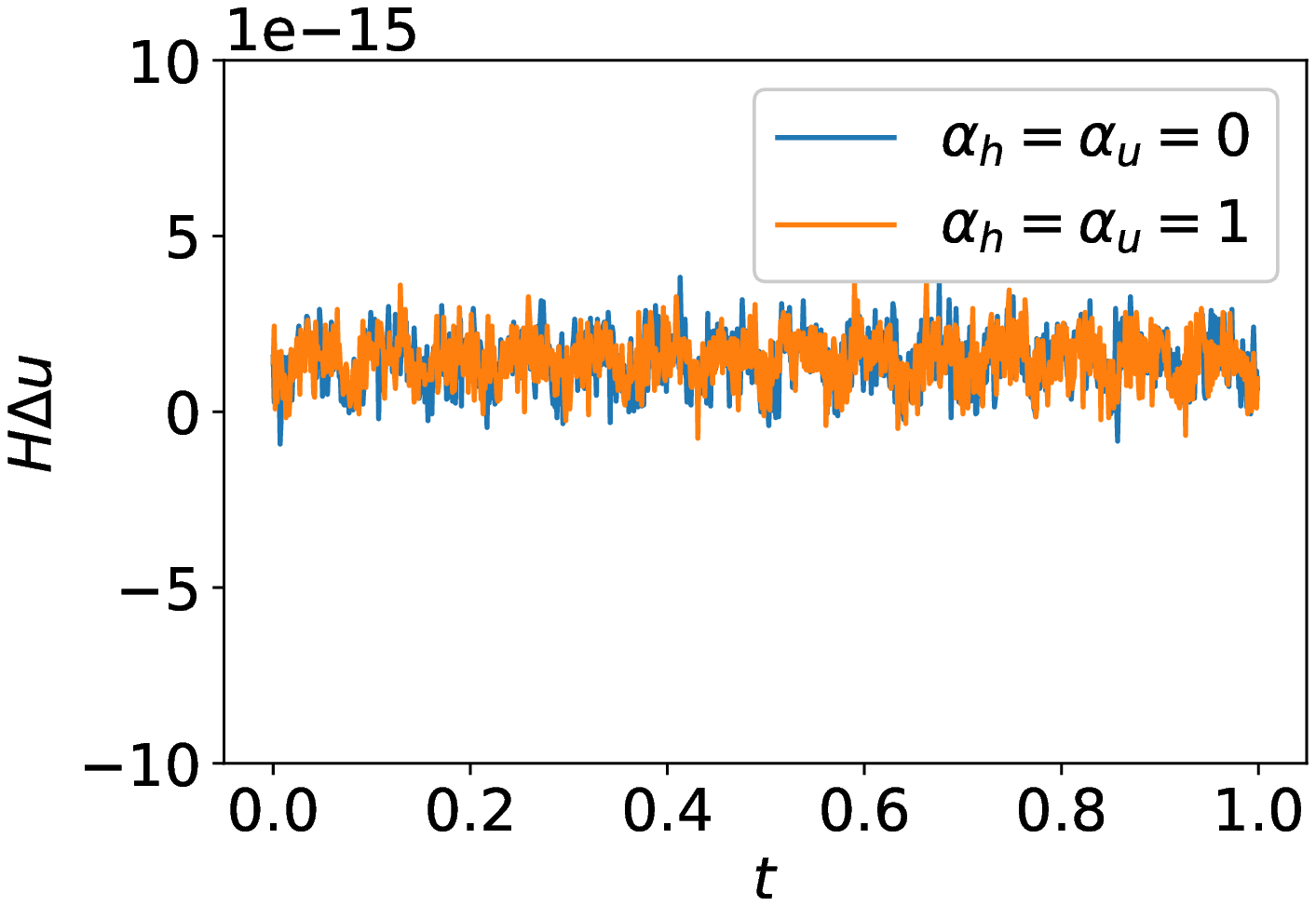}
         \caption{}
     \end{subfigure}
     \begin{subfigure}[b]{0.325\textwidth}
         \includegraphics[width=\textwidth]{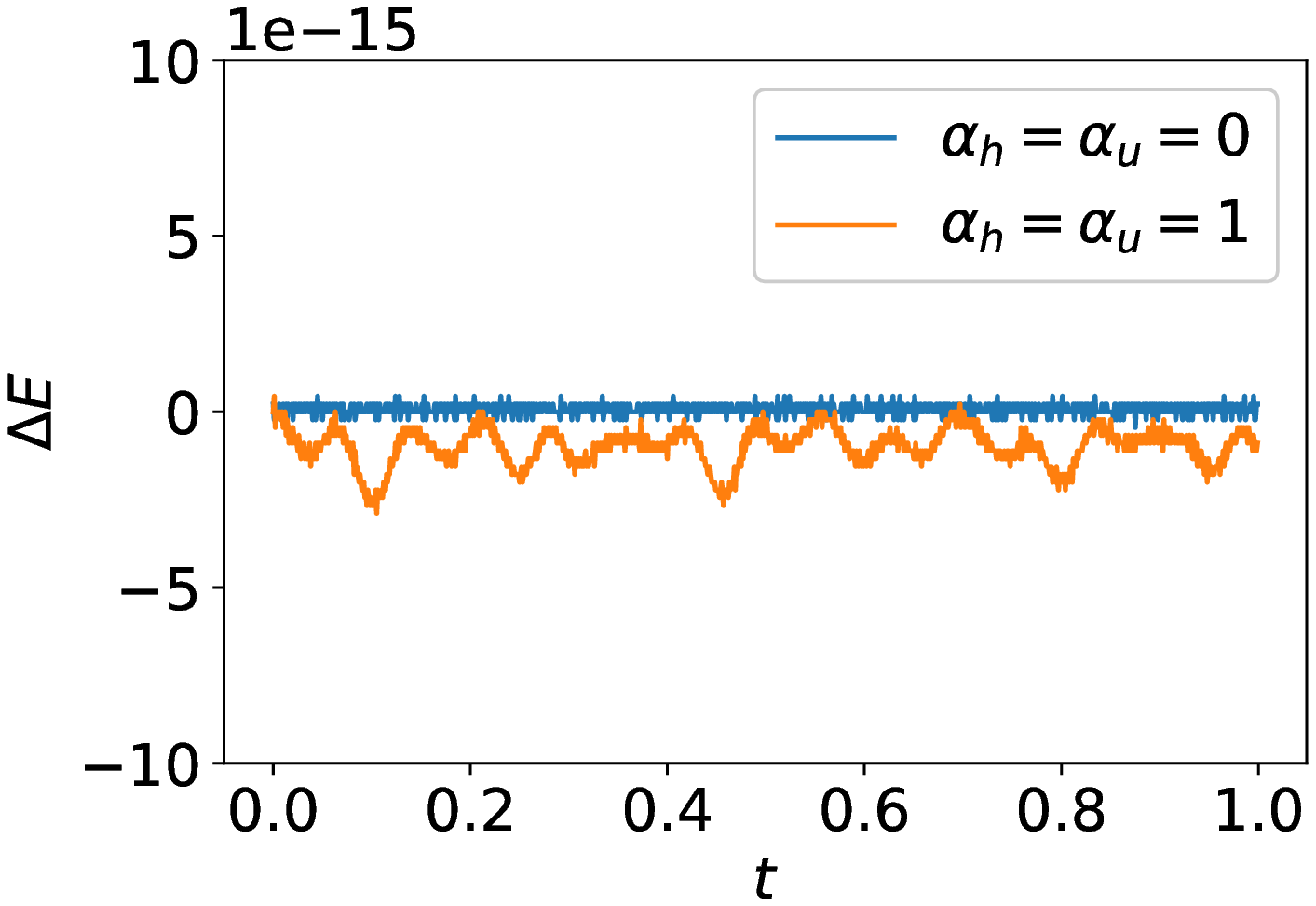}
          \caption{}
     \end{subfigure}
     \caption{Conservation test for (a) mass, (b) momentum and (c) energy achieving machine precision tolerance for the periodic boundary conditions problem with $U = 0.2$ for different upwind parameters $\alpha_h$ and $\alpha_u$.} \label{conservationfig2}
\end{figure}

\subsection{Periodic boundary conditions problem convergence test} \label{periodicnumericalexperiment}
We consider the periodic boundary conditions problem described in \cref{numericalexperiments}. The domain of the problem $\Omega$ is discretized into $N = 10,20,40,80$ elements, and polynomials of degree $P = 1,2,3,4$ are used to compute the numerical approximation. Similar to the previous section, we consider two different choices of upwind parameters $\alpha_h = \alpha_u = 0$ and $\alpha_h = \alpha_u = 1$, and the numerical approximation is evolved using the classical explicit fourth order accurate Runge-Kutta method until the final time $T = 0.1$. The time step is computed using the following:
\begin{equation} \label{CFL}
    \Delta t = \operatorname{CFL}\left(\frac{\Delta x}{P + 1}\right)^2, \quad \operatorname{CFL} = 0.1
\end{equation}
where $\Delta x$ is the element length. 

To investigate the convergence of the numerical errors, at the final time $T = 0.1$, we compute the $L^2$ error $\|e\|_{L^2(\Omega)} \approx \|\mathbf{u} - \mathbf{u}_{\text{ex}}\|_{\mathbf{M}}$, where $\mathbf{u}$ and $\mathbf{u}_{\text{ex}}$ are the numerical and exact solutions respectively. We have omitted the convergence plots of the height $h$ for brevity. The convergence plots and rates are depicted in Fig. \ref{periodicconvergenceplot1}-\ref{periodicconvergenceplot2} for $u$ and are given in Table \ref{periodicconvergencetable}. Overall, we observe that the numerical approximation for the velocity $u$ is $(P+1)$th order accurate when $P$ is even and $P$th order accurate when $P$ is odd. For the height $h$, it can be seen that when $\alpha_h = \alpha_u = 0$, the numerical approximation is $P$th order accurate, and we get improved convergence rates when $\alpha_h = \alpha_u = 1$. 

\begin{remark}
The time step is proportional to $(\Delta x)^2$ due to the presence of higher order spatial derivatives and the use of an explicit time stepping scheme. We note that employing an appropriate implicit time stepping scheme will eliminate this restriction, but this is not the main focus of this paper.
\end{remark}

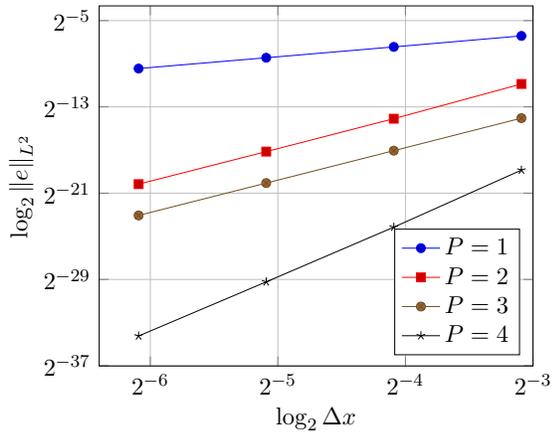
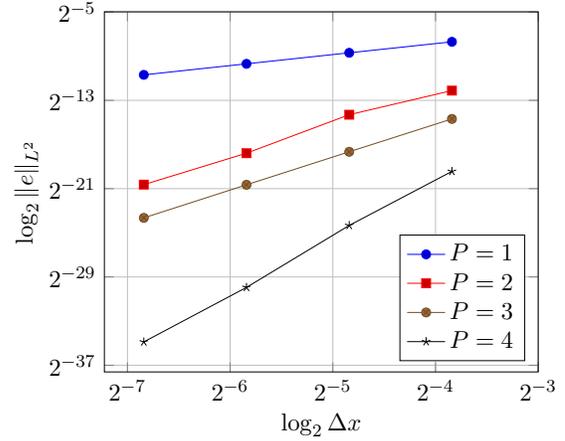
\begin{figure}[H]
  \begin{subfigure}[t]{.45\textwidth}
    \centering
    \resizebox{\textwidth}{!}{\begin{tikzpicture}
\begin{loglogaxis}[
xlabel={$\log_2 \Delta x$},
ylabel={$\log_2 \|e\|_{L^2}$},
  log basis y = {2},
  log basis x = {2},
  legend entries = {$P = 1$, $P = 2$, $P = 3$, $P = 4$},
  legend pos=south east,
  grid=major,
  xmax = 0.125,
]
\addplot table[x index=0,y index=1,col sep=comma] {periodic_u_zero.csv};
\addplot table[x index=0,y index=2,col sep=comma] {periodic_u_zero.csv};
\addplot table[x index=0,y index=3,col sep=comma] {periodic_u_zero.csv};
\addplot table[x index=0,y index=4,col sep=comma] {periodic_u_zero.csv};
\end{loglogaxis}
\end{tikzpicture}}
    \caption{}
  \end{subfigure} \hfill
  \begin{subfigure}[t]{.45\textwidth}
    \centering
    \resizebox{\textwidth}{!}{\begin{tikzpicture}
\begin{loglogaxis}[
xlabel={$\log_2 \Delta x$},
ylabel={$\log_2 \|e\|_{L^2}$},
  log basis y = {2},
  log basis x = {2},
  legend entries = {$P = 1$, $P = 2$, $P = 3$, $P = 4$},
  legend pos=south east,
  grid=major,
  xmax = 0.125,
]
\addplot table[x index=0,y index=1,col sep=comma] {periodic_u.csv};
\addplot table[x index=0,y index=2,col sep=comma] {periodic_u.csv};
\addplot table[x index=0,y index=3,col sep=comma] {periodic_u.csv};
\addplot table[x index=0,y index=4,col sep=comma] {periodic_u.csv};
\end{loglogaxis}
\end{tikzpicture}}
    \caption{}
  \end{subfigure}
\caption{Periodic boundary conditions problem convergence plots of u for (a) $U = 0$ and (b) $U = 0.2$ with upwind parameters $\alpha_h = \alpha_u = 0$ for different polynomial degrees $P$.} \label{periodicconvergenceplot1}
\end{figure}

\begin{figure}[H]
  \begin{subfigure}[t]{.45\textwidth}
    \centering
    \resizebox{\textwidth}{!}{\begin{tikzpicture}
\begin{loglogaxis}[
xlabel={$\log_2 \Delta x$},
ylabel={$\log_2 \|e\|_{L^2}$},
  log basis y = {2},
  log basis x = {2},
  legend entries = {$P = 1$, $P = 2$, $P = 3$, $P = 4$},
  legend pos=south east,
  grid=major,
  xmax = 0.125,
]
\addplot table[x index=0,y index=1,col sep=comma] {periodic_u_upwind_zero.csv};
\addplot table[x index=0,y index=2,col sep=comma] {periodic_u_upwind_zero.csv};
\addplot table[x index=0,y index=3,col sep=comma] {periodic_u_upwind_zero.csv};
\addplot table[x index=0,y index=4,col sep=comma] {periodic_u_upwind_zero.csv};
\end{loglogaxis}
\end{tikzpicture}}
    \caption{}
  \end{subfigure} \hfill
  \begin{subfigure}[t]{.45\textwidth}
    \centering
    \resizebox{\textwidth}{!}{\begin{tikzpicture}
\begin{loglogaxis}[
xlabel={$\log_2 \Delta x$},
ylabel={$\log_2 \|e\|_{L^2}$},
  log basis y = {2},
  log basis x = {2},
  legend entries = {$P = 1$, $P = 2$, $P = 3$, $P = 4$},
  legend pos=south east,
  grid=major,
  xmax = 0.125,
]
\addplot table[x index=0,y index=1,col sep=comma] {periodic_u_upwind.csv};
\addplot table[x index=0,y index=2,col sep=comma] {periodic_u_upwind.csv};
\addplot table[x index=0,y index=3,col sep=comma] {periodic_u_upwind.csv};
\addplot table[x index=0,y index=4,col sep=comma] {periodic_u_upwind.csv};
\end{loglogaxis}
\end{tikzpicture}}
    \caption{}
  \end{subfigure}
\caption{Periodic boundary conditions problem convergence plots of u for (a) $U = 0$ and (b) $U = 0.2$ with upwind parameters $\alpha_h = \alpha_u = 1$ for different polynomial degrees $P$.} \label{periodicconvergenceplot2}
\end{figure}
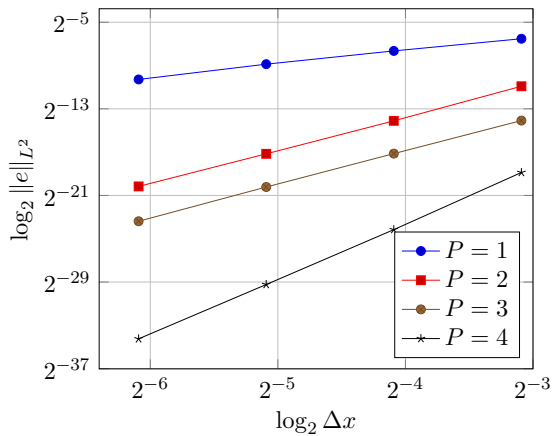
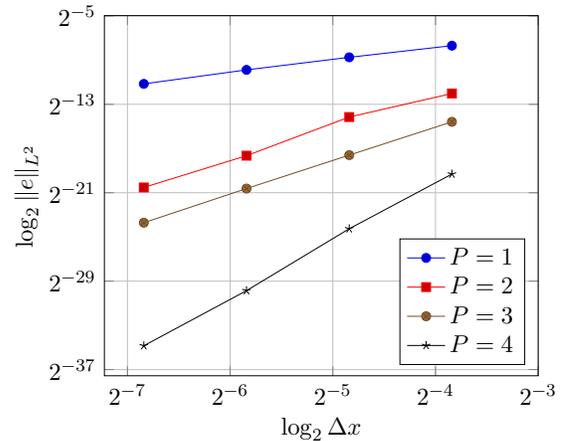

\begin{table}[H]
\centering
\begin{tabular}{|c|cc|cc|cc|cc|}
\hline
\multirow{2}{*}{$P$} & \multicolumn{2}{c|}{$U = 0$, $\alpha_h = \alpha_u = 0$} & \multicolumn{2}{c|}{$U = 0.2$, $\alpha_h = \alpha_u = 0$} & \multicolumn{2}{c|}{$U = 0$, $\alpha_h = \alpha_u = 1$} & \multicolumn{2}{c|}{$U = 0.2$, $\alpha_h = \alpha_u = 1$} \\ \cline{2-9} 
                     & \multicolumn{1}{c|}{$h$}              & $u$             & \multicolumn{1}{c|}{$h$}               & $u$              & \multicolumn{1}{c|}{$h$}              & $u$             & \multicolumn{1}{c|}{$h$}               & $u$              \\ \hline
1                    & \multicolumn{1}{c|}{$0.980$}          & $1.009$         & \multicolumn{1}{c|}{$0.982$}           & $0.998$          & \multicolumn{1}{c|}{$1.964$}          & $1.250$         & \multicolumn{1}{c|}{$1.952$}           & $1.151$          \\ \hline
2                    & \multicolumn{1}{c|}{$1.989$}          & $3.091$         & \multicolumn{1}{c|}{$2.232$}           & $2.904$          & \multicolumn{1}{c|}{$1.956$}          & $3.080$         & \multicolumn{1}{c|}{$2.478$}           & $2.900$          \\ \hline
3                    & \multicolumn{1}{c|}{$2.991$}          & $3.006$         & \multicolumn{1}{c|}{$2.990$}           & $2.988$          & \multicolumn{1}{c|}{$3.040$}          & $3.098$         & \multicolumn{1}{c|}{$4.676$}           & $3.041$          \\ \hline
4                    & \multicolumn{1}{c|}{$3.993$}          & $5.121$         & \multicolumn{1}{c|}{$4.129$}           & $5.195$          & \multicolumn{1}{c|}{$3.969$}          & $5.116$         & \multicolumn{1}{c|}{$4.139$}           & $5.219$          \\ \hline
\end{tabular} \caption{Periodic boundary conditions problem convergence rates of $h$ and $u$ for $U = 0$ and $0.2$ with $\alpha_h = \alpha_u = 0$, $\alpha_h = \alpha_u = 1$ for different polynomial degrees $P$.}
\label{periodicconvergencetable}
\end{table}

\subsection{Initial boundary value problem convergence test} We consider the IBVP described in \cref{numericalexperiments}. The numerical approximation is computed using the same discretization parameters as in \cref{periodicnumericalexperiment}.

We investigate the convergence of the numerical errors by computing the $L^2$ error at the final time $T = 0.1$. Fig. \ref{convplotIBVP1}-\ref{convplotIBVP2} and Table \ref{IBVPconvergencetable} depict the convergence plots and rates respectively. It is observed overall that the numerical approximation for the velocity $u$ is $P$th order accurate when $P$ is odd and $(P+1)$th order accurate when $P$ is even. The numerical approximation for the height $h$ is $P$th order accurate when $\alpha_h = \alpha_u = 0$, and better convergence rates are obtained when $\alpha_h = \alpha_u = 1$. These observed convergence rates are consistent with the periodic boundary conditions case.
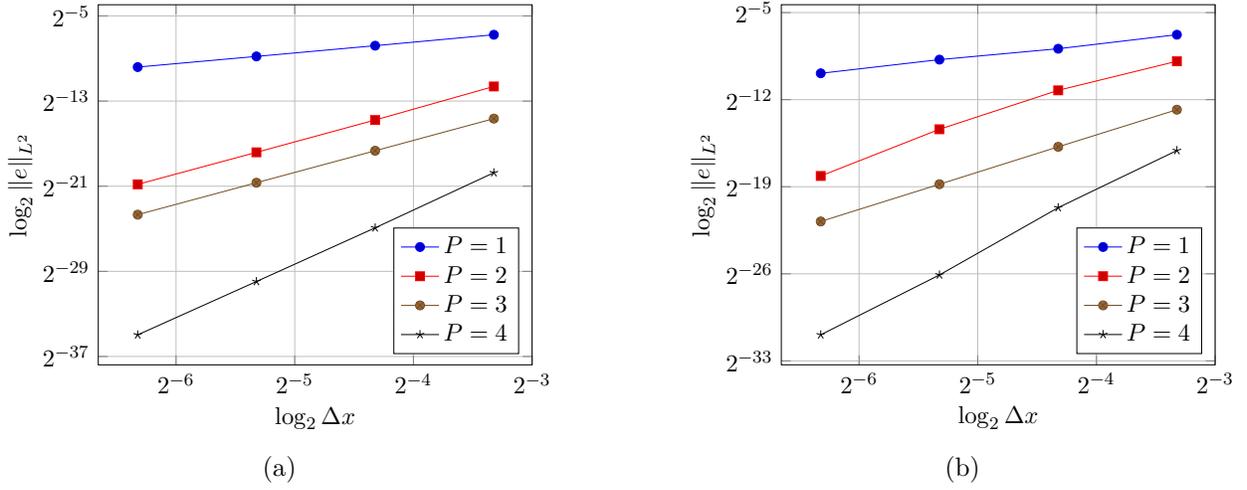
\begin{figure}[H]
  \begin{subfigure}[t]{.45\textwidth}
    \centering
    \resizebox{\textwidth}{!}{\begin{tikzpicture}
\begin{loglogaxis}[
xlabel={$\log_2 \Delta x$},
ylabel={$\log_2 \|e\|_{L^2}$},
  log basis y = {2},
  log basis x = {2},
  legend entries = {$P = 1$, $P = 2$, $P = 3$, $P = 4$},
  legend pos=south east,
  grid=major,
  xmax = 0.125,
]
\addplot table[x index=0,y index=1,col sep=comma] {IBVP_u_zero.csv};
\addplot table[x index=0,y index=2,col sep=comma] {IBVP_u_zero.csv};
\addplot table[x index=0,y index=3,col sep=comma] {IBVP_u_zero.csv};
\addplot table[x index=0,y index=4,col sep=comma] {IBVP_u_zero.csv};
\end{loglogaxis}
\end{tikzpicture}}
    \caption{}
  \end{subfigure} \hfill
  \begin{subfigure}[t]{.45\textwidth}
    \centering
    \resizebox{\textwidth}{!}{\begin{tikzpicture}
\begin{loglogaxis}[
xlabel={$\log_2 \Delta x$},
ylabel={$\log_2 \|e\|_{L^2}$},
  log basis y = {2},
  log basis x = {2},
  legend entries = {$P = 1$, $P = 2$, $P = 3$, $P = 4$},
  legend pos=south east,
  grid=major,
  xmax = 0.125,
]
\addplot table[x index=0,y index=1,col sep=comma] {IBVP_u.csv};
\addplot table[x index=0,y index=2,col sep=comma] {IBVP_u.csv};
\addplot table[x index=0,y index=3,col sep=comma] {IBVP_u.csv};
\addplot table[x index=0,y index=4,col sep=comma] {IBVP_u.csv};
\end{loglogaxis}
\end{tikzpicture}}
    \caption{}
  \end{subfigure}
\caption{Initial boundary value problem convergence plots of u for (a) $U = 0$ and (b) $U = 0.2$ with upwind parameters $\alpha_h = \alpha_u = 0$ for different polynomial degrees $P$.} \label{convplotIBVP1}
\end{figure}

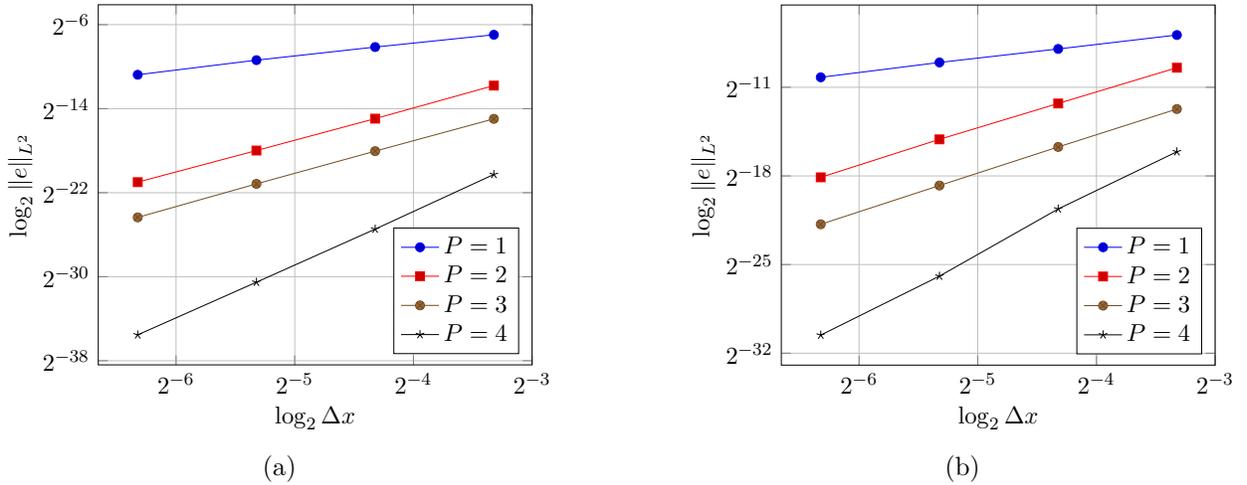
\begin{figure}[H]
  \begin{subfigure}[t]{.45\textwidth}
    \centering
    \resizebox{\textwidth}{!}{\begin{tikzpicture}
\begin{loglogaxis}[
xlabel={$\log_2 \Delta x$},
ylabel={$\log_2 \|e\|_{L^2}$},
  log basis y = {2},
  log basis x = {2},
  legend entries = {$P = 1$, $P = 2$, $P = 3$, $P = 4$},
  legend pos=south east,
  grid=major,
  xmax = 0.125,
]
\addplot table[x index=0,y index=1,col sep=comma] {IBVP_u_upwind_zero.csv};
\addplot table[x index=0,y index=2,col sep=comma] {IBVP_u_upwind_zero.csv};
\addplot table[x index=0,y index=3,col sep=comma] {IBVP_u_upwind_zero.csv};
\addplot table[x index=0,y index=4,col sep=comma] {IBVP_u_upwind_zero.csv};
\end{loglogaxis}
\end{tikzpicture}}
    \caption{}
  \end{subfigure} \hfill
  \begin{subfigure}[t]{.45\textwidth}
    \centering
    \resizebox{\textwidth}{!}{\begin{tikzpicture}
\begin{loglogaxis}[
xlabel={$\log_2 \Delta x$},
ylabel={$\log_2 \|e\|_{L^2}$},
  log basis y = {2},
  log basis x = {2},
  legend entries = {$P = 1$, $P = 2$, $P = 3$, $P = 4$},
  legend pos=south east,
  grid=major,
  xmax = 0.125,
]
\addplot table[x index=0,y index=1,col sep=comma] {IBVP_u_upwind.csv};
\addplot table[x index=0,y index=2,col sep=comma] {IBVP_u_upwind.csv};
\addplot table[x index=0,y index=3,col sep=comma] {IBVP_u_upwind.csv};
\addplot table[x index=0,y index=4,col sep=comma] {IBVP_u_upwind.csv};
\end{loglogaxis}
\end{tikzpicture}}
    \caption{}
  \end{subfigure}
\caption{Initial boundary value problem convergence plots of u for (a) $U = 0$ and (b) $U = 0.2$ with upwind parameters $\alpha_h = \alpha_u = 1$ for different polynomial degrees $P$.} \label{convplotIBVP2}
\end{figure}

\begin{table}[H]
\centering
\begin{tabular}{|c|cc|cc|cc|cc|}
\hline
\multirow{2}{*}{$P$} & \multicolumn{2}{c|}{$U = 0$, $\alpha_h = \alpha_u = 0$} & \multicolumn{2}{c|}{$U = 0.2$, $\alpha_h = \alpha_u = 0$} & \multicolumn{2}{c|}{$U = 0$, $\alpha_h = \alpha_u = 1$} & \multicolumn{2}{c|}{$U = 0.2$, $\alpha_h = \alpha_u = 1$} \\ \cline{2-9} 
                     & \multicolumn{1}{c|}{$h$}              & $u$             & \multicolumn{1}{c|}{$h$}               & $u$              & \multicolumn{1}{c|}{$h$}              & $u$             & \multicolumn{1}{c|}{$h$}               & $u$              \\ \hline
1                    & \multicolumn{1}{c|}{$0.984$}          & $1.012$         & \multicolumn{1}{c|}{$0.962$}           & $1.018$          & \multicolumn{1}{c|}{$1.632$}          & $1.267$         & \multicolumn{1}{c|}{$1.898$}           & $1.109$          \\ \hline
2                    & \multicolumn{1}{c|}{$1.993$}          & $3.064$         & \multicolumn{1}{c|}{$2.201$}           & $3.078$          & \multicolumn{1}{c|}{$1.972$}          & $3.061$         & \multicolumn{1}{c|}{$2.384$}           & $2.880$          \\ \hline
3                    & \multicolumn{1}{c|}{$2.992$}          & $3.008$         & \multicolumn{1}{c|}{$2.995$}           & $2.996$          & \multicolumn{1}{c|}{$3.039$}          & $3.126$         & \multicolumn{1}{c|}{$3.811$}           & $3.034$          \\ \hline
4                    & \multicolumn{1}{c|}{$3.996$}          & $5.074$         & \multicolumn{1}{c|}{$4.187$}           & $4.985$          & \multicolumn{1}{c|}{$3.981$}          & $5.087$         & \multicolumn{1}{c|}{$4.244$}           & $4.874$          \\ \hline
\end{tabular}
\caption{Initial boundary value problem convergence rates of $h$ and $u$ for $U = 0$ and $0.2$ with $\alpha_h = \alpha_u = 0$, $\alpha_h = \alpha_u = 1$ for different polynomial degrees $P$.}
\label{IBVPconvergencetable}
\end{table}

\subsection{Gaussian initial condition test}

We consider the linearized Serre equations \eqref{linearizedserre} with the initial conditions
\begin{equation} \label{gaussianinitialcond}
h(x,0) = \frac{1}{5}e^{-25x^2}, \quad u(x,0) = 0,
\end{equation}
and the periodic boundary conditions \eqref{eq:peridic_boundary_conditions} in the bounded domain $\Omega = [-5,5]$. We choose the model parameters $g = 9.8$, $H = 1.0$, $U = 0.2$.

To illustrate the effectiveness of high order accuracy, we compute the numerical solution using three different discretization parameters $P = 2$, $N = 48$, $P = 2$, $N = 96$, and $P = 8$, $N = 16$, where $P$ and $N$ denote the polynomial degree and the number of elements respectively. The classical explicit fourth order accurate Runge-Kutta method with the time step $\Delta t$ determined by \eqref{CFL} is used to evolve the numerical solution until the final time $T = 6$. 

Snapshots of the numerical solution for $h$ and $u$ are depicted in Figure \ref{gaussian1} and \ref{gaussian2} respectively. It is observed that when $P = 2$, $N = 48$, the numerical solution has visible spurious oscillations. On the other hand, there are none such oscillations when $P = 8, N = 16$, even though they have the same degrees of freedom and take roughly the same amount of computation time. Doubling the degrees of freedom of $P = 2, N = 48$ by mesh refinement gives $P = 2, N = 96$, and the numerical solution in this case is similar to the high order case $P = 8, N = 16$ but it takes significantly higher computation time.

\begin{figure}[H] 
     \centering
     \begin{subfigure}[b]{0.325\textwidth}
         \includegraphics[width=\textwidth]{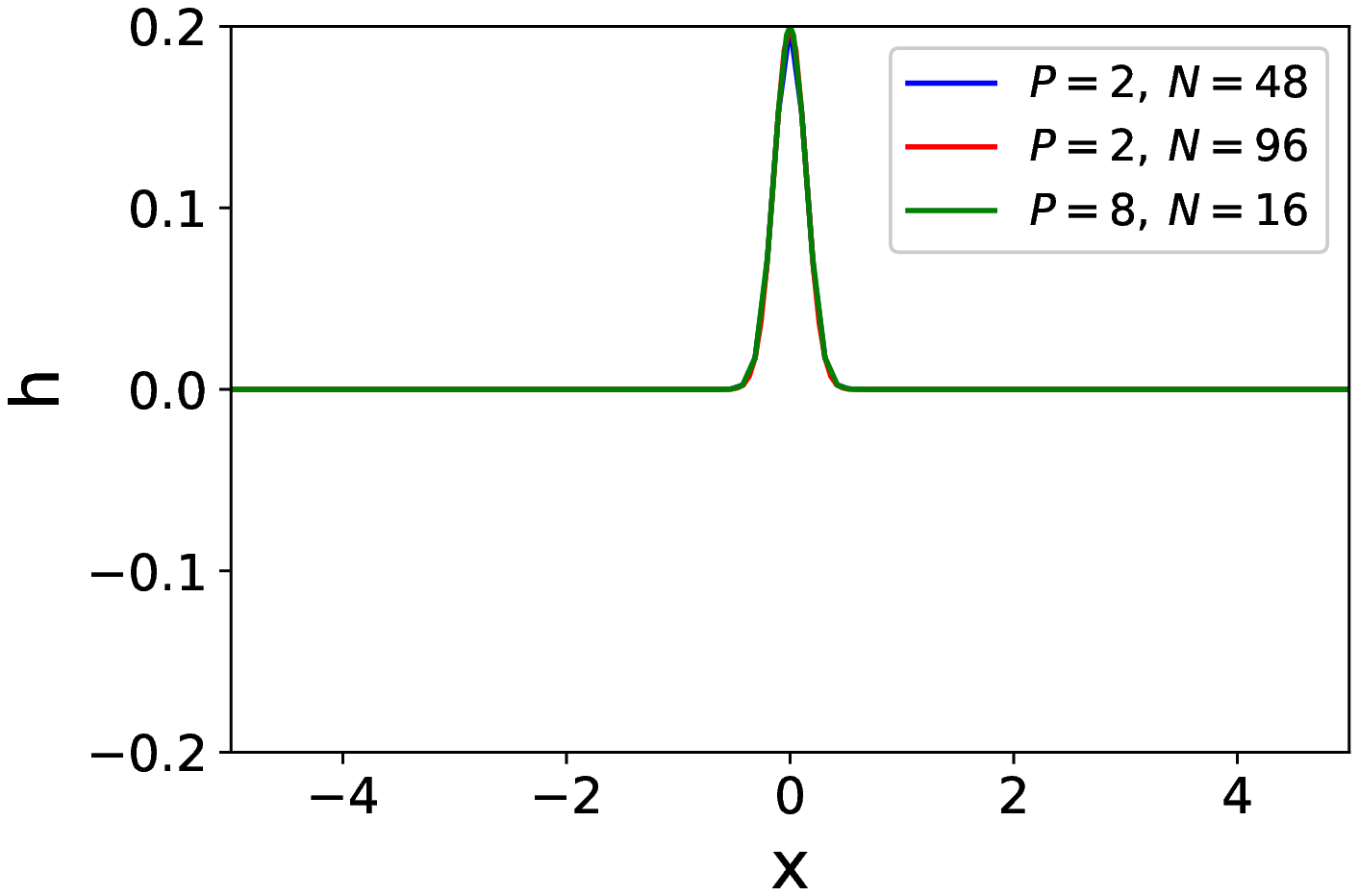}
         \caption{}
     \end{subfigure}
    \begin{subfigure}[b]{0.325\textwidth}
         \includegraphics[width=\textwidth]{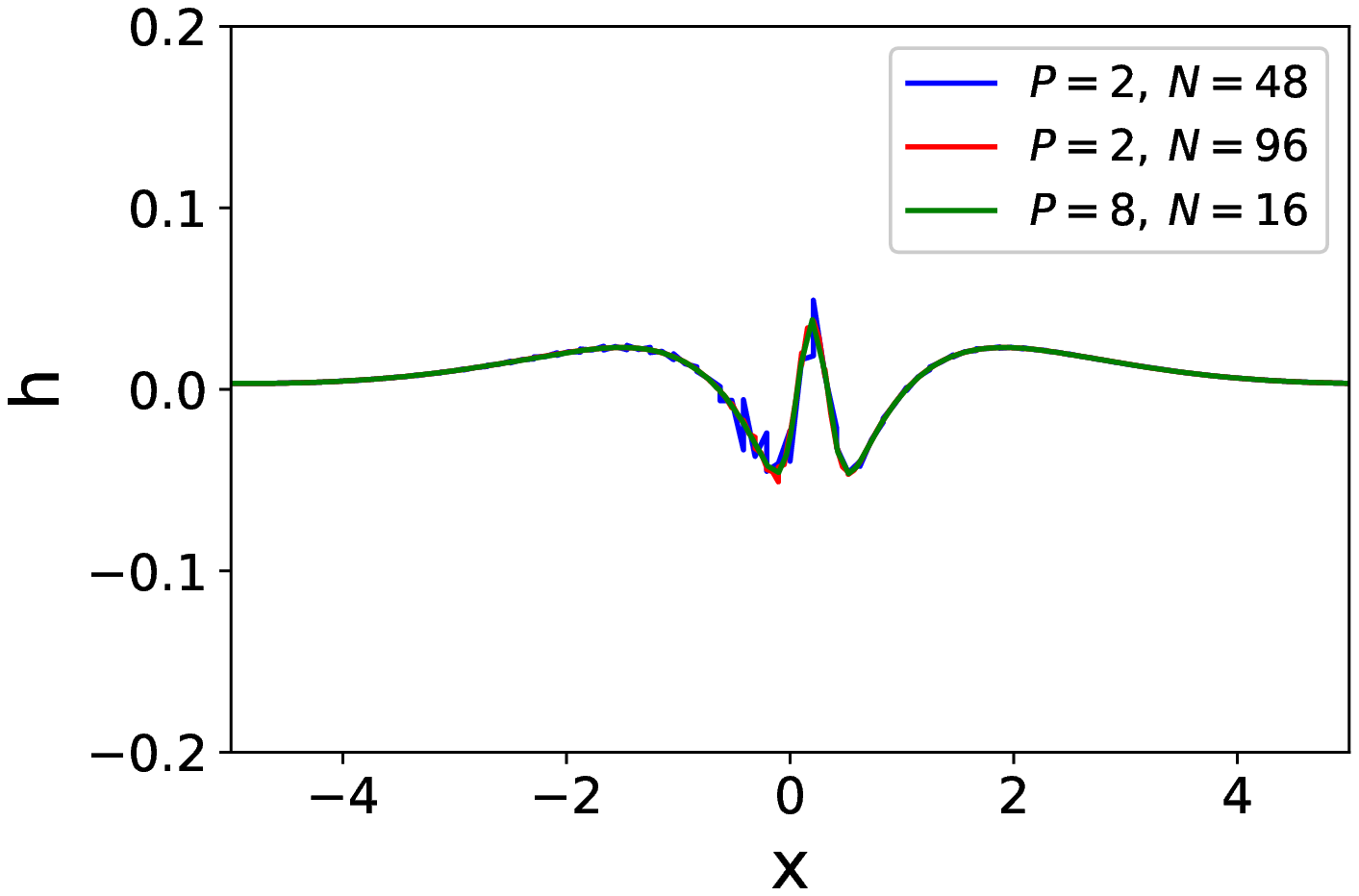}
         \caption{}
     \end{subfigure}
     \begin{subfigure}[b]{0.325\textwidth}
         \includegraphics[width=\textwidth]{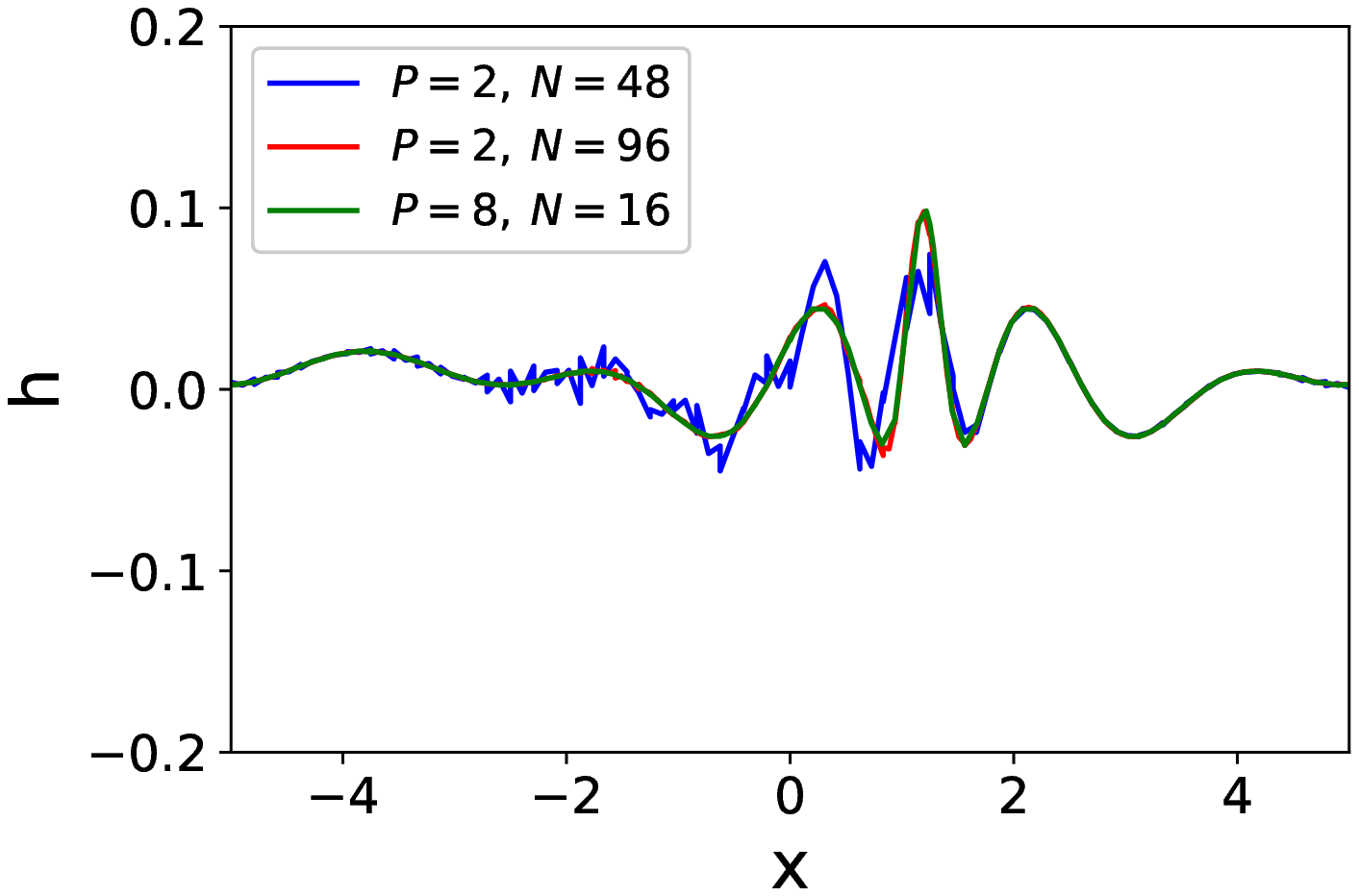}
         \caption{}
     \end{subfigure}
     \caption{Height $h$ numerical solution of the linearized Serre equations with the Gaussian initial condition \eqref{gaussianinitialcond} and the periodic boundary conditions \eqref{eq:peridic_boundary_conditions} at (a) $t = 0$, (b) $t = 1$, and (c) $t = 6$ for different polynomial degrees $P$ and number of elements $N$. } \label{gaussian1}
\end{figure}

\begin{figure}[H] 
     \centering
     \begin{subfigure}[b]{0.325\textwidth}
         \includegraphics[width=\textwidth]{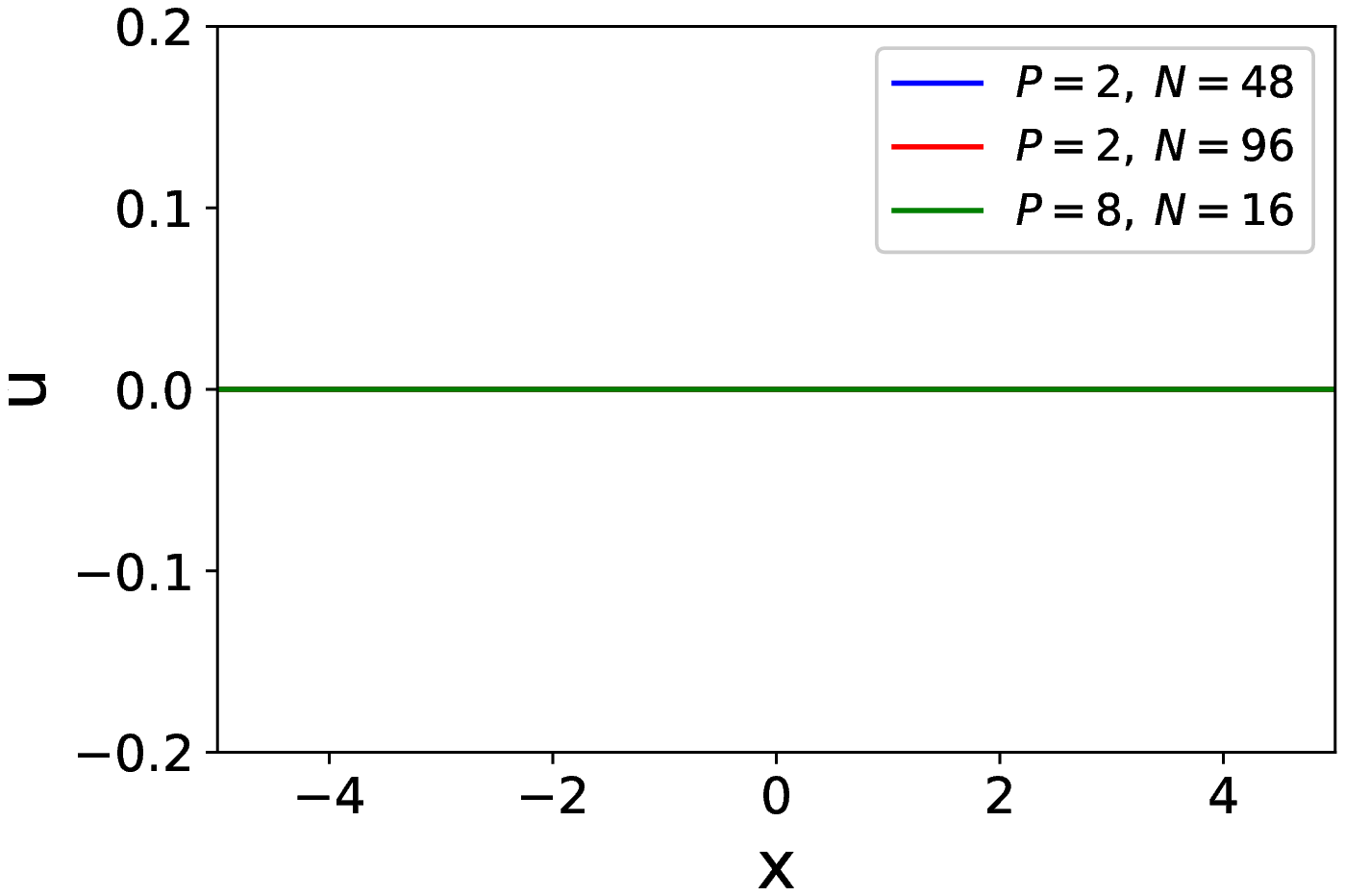}
          \caption{}
     \end{subfigure}
     \begin{subfigure}[b]{0.325\textwidth}
         \includegraphics[width=\textwidth]{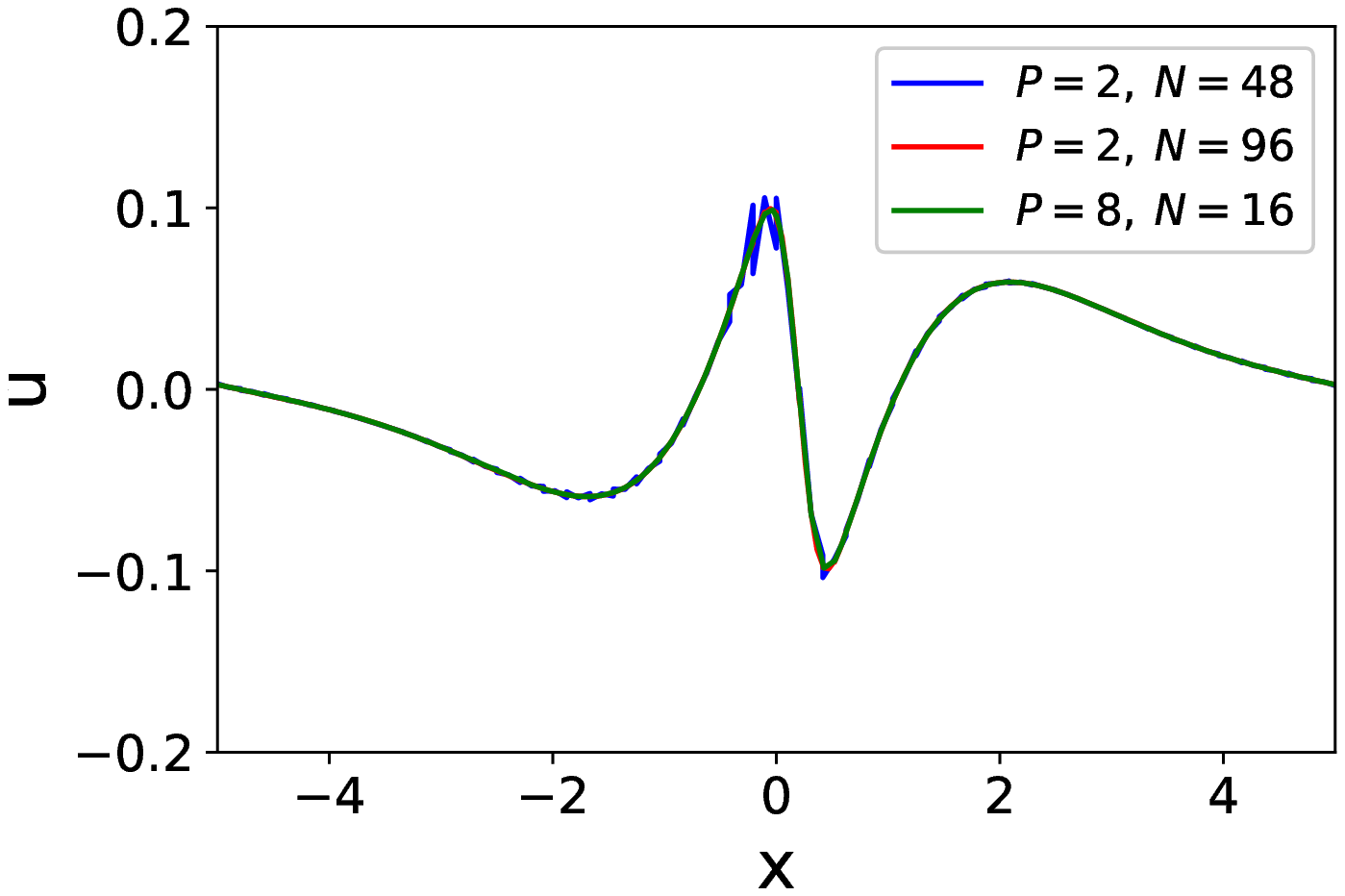}
         \caption{}
     \end{subfigure}
     \begin{subfigure}[b]{0.325\textwidth}
         \includegraphics[width=\textwidth]{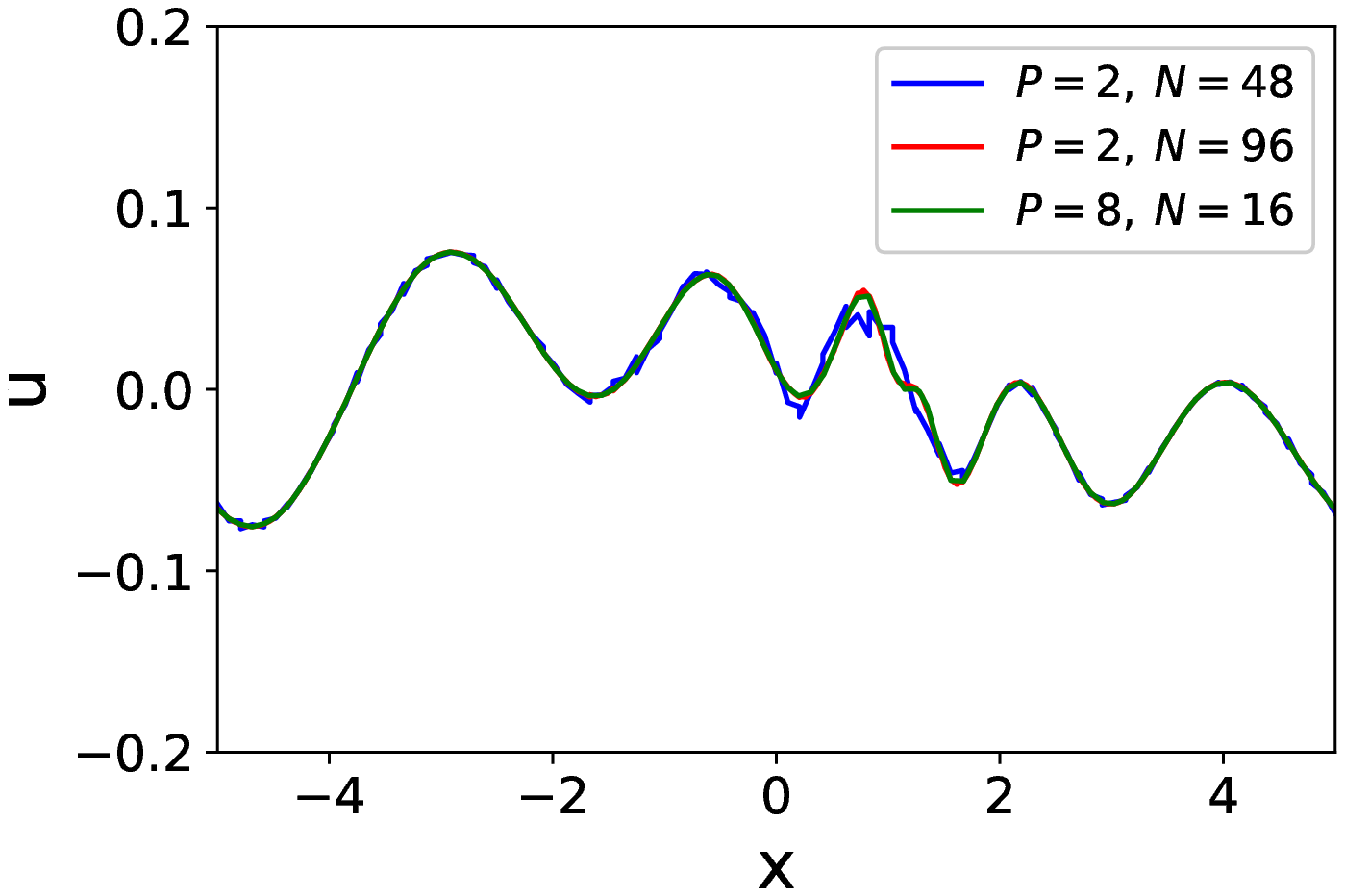}
          \caption{}
     \end{subfigure}
     \caption{Velocity $u$ numerical solution of the linearized Serre equations with the Gaussian initial condition \eqref{gaussianinitialcond} and the periodic boundary conditions \eqref{eq:peridic_boundary_conditions} at (a) $t = 0$, (b) $t = 1$, and (c) $t = 6$ for different polynomial degrees $P$ and number of elements $N$. } \label{gaussian2}
\end{figure}

\section{Conclusion}
In this paper, we have derived and analyzed well-posed boundary conditions for the linearized Serre equations. The analysis is based on the energy method and it identifies the number, location, and form of the boundary conditions so that the IBVP is well-posed. In particular, when the background flow velocity is nonzero it was shown that we need a total of four boundary conditions, specifically three boundary conditions at the inflow boundary and one boundary condition at the outflow boundary, to bound the solution in the energy norm. When the background flow velocity is zero only two boundary conditions are needed, one boundary condition at each end boundary of a 1D interval, to ensure well-posedness. Furthermore, to couple adjacent elements we derived well-posed interface conditions that ensure the conservation of energy, mass, and linear momentum. 

We have developed a provably stable DGSEM for the linearized Serre equations of arbitrary order of accuracy. The discretization is based on discontinuous Galerkin spectral derivative operators that satisfy the SBP properties for first, second, and third order derivatives. These operators are used in combination with the SAT method to impose the interface and boundary conditions numerically in a stable manner. With appropriately chosen penalty parameters, we have shown that the proposed numerical interface and boundary treatments emulate the well-posedness properties in the continuous analysis. A priori error estimates were also derived in the energy norm. Numerical experiments have been presented verifying the theoretical results and demonstrating the efficiency of high order DGSEM for resolving highly oscillatory dispersive wave modes.

The numerical method developed in this paper has been extended for solving the nonlinear Serre equations, and it will be published in a forthcoming paper. An obvious extension of this paper is to develop a provably stable numerical method for solving the Serre equations in two spatial dimensions.

\appendix
\section{Eigen-decomposition} \label{eigendecompositionappendix}

We aim to eigen-decompose
\begin{equation*}
    \mathbf{A} = \left[\begin{matrix}- \frac{g U}{2} & - \frac{g H}{2} & 0 & 0 & 0\\- \frac{g H}{2} & - \frac{H U}{2} & 0 & \frac{H^{3} U}{6} & \frac{H^{3}}{6}\\0 & 0 & - \frac{H^{3} U}{6} & 0 & 0\\0 & \frac{H^{3} U}{6} & 0 & 0 & 0\\0 & \frac{H^{3}}{6} & 0 & 0 & 0\end{matrix}\right]
\end{equation*}
in \eqref{matrixA}. We introduce the transformation matrix
\begin{equation*}
\mathbf{P} = \left[\begin{matrix}1 & 0 & 0 & 0 & \frac{3 g}{H^{2}}\\0 & 1 & 0 & 0 & 0\\0 & 0 & 1 & 0 & 0\\0 & 0 & 0 & 1 & - U\\0 & 0 & 0 & 0 & 1\end{matrix}\right],
\end{equation*}
and then we utilize this matrix to transform the matrix $\mathbf{A}$ as follows:
\begin{equation} \label{transformedmatrix}
    \mathbf{P}\mathbf{A}\mathbf{P}^{\top} = \left[\begin{matrix}- \frac{g U}{2} & 0 & 0 & 0 & 0\\0 & - \frac{H U}{2} & 0 & 0 & \frac{H^{3}}{6}\\0 & 0 & - \frac{H^{3} U}{6} & 0 & 0\\0 & 0 & 0 & 0 & 0\\0 & \frac{H^{3}}{6} & 0 & 0 & 0\end{matrix}\right].
\end{equation}
Applying the eigenvalue decomposition to \eqref{transformedmatrix} yields
\begin{equation*} \label{eigendecompose}
    \mathbf{P}\mathbf{A}\mathbf{P}^{\top} = \mathbf{Q}\mathbf{\Lambda}\mathbf{Q}^\top,
\end{equation*}
where
\begin{align*} \label{eigenvalues}
    \mathbf{\Lambda} &= \operatorname{diag}\left(\lambda_1, \lambda_2, \lambda_3, \lambda_4, \lambda_5\right) \\
&= \left[\begin{matrix}0 & 0 & 0 & 0 & 0\\0 & - \frac{H^{3} U}{6} & 0 & 0 & 0\\0 & 0 & - \frac{g U}{2} & 0 & 0\\0 & 0 & 0 & - \frac{H U}{4} - \frac{H \sqrt{4 H^{4} + 9 U^{2}}}{12} & 0\\0 & 0 & 0 & 0 & - \frac{H U}{4} + \frac{H \sqrt{4 H^{4} + 9 U^{2}}}{12}\end{matrix}\right].
\end{align*}
is the diagonal matrix of the eigenvalues of $\mathbf{P}\mathbf{A}\mathbf{P}^{\top}$ and $\mathbf{Q}$ is the orthogonal matrix containing the corresponding eigenvectors.
Using this obtained decomposition, the boundary term can be further rewritten as
\begin{equation*}
    \operatorname{BT} = \mathbf{w}^{\top} \mathbf{\Lambda} \mathbf{w} = \sum_{i=1}^5 	\lambda_i \mathbf{w}_i^2
\end{equation*}
where the vector $\mathbf{w}$ is given by
\begin{equation*}
    \mathbf{w} = \left(\mathbf{P}^{-1} \mathbf{Q} \right)^{\top}.
\end{equation*}

\section{Proof of \cref{interfacesatlemma}} \label{interfacesatappendix}

We observe that the penalty terms in \eqref{interfacesat1} can be rewritten as
\begin{align*}
\tau_{11}H\mathbf{M}^{-1}\widetilde{\mathbf{B}}\mathbf{u} + \tau_{12}U\mathbf{M}^{-1}\widetilde{\mathbf{B}}\mathbf{h} &= \frac{1}{2}\mathbf{M}^{-1}\widetilde{\mathbf{B}}(H\mathbf{u} + U\mathbf{h}) \\ &= -\left(\mathbf{D} - \frac{1}{2}\mathbf{M}^{-1}\widetilde{\mathbf{B}}\right)(H\mathbf{u} + U\mathbf{h}) + \mathbf{D}(H\mathbf{u} + U\mathbf{h}) \\
&= -\widetilde{\mathbf{D}}(H\mathbf{u} + U\mathbf{h}) + \mathbf{D}(H\mathbf{u} + U\mathbf{h}),
\end{align*}
and hence \eqref{interfacesat1} becomes
\begin{equation*}
    \frac{d\mathbf{h}}{dt} + \widetilde{\mathbf{D}}(H\mathbf{u} + U\mathbf{h}) + \alpha_h \mathbf{M}^{-1}\widetilde{\mathbf{B}}^{\top} \widetilde{\mathbf{B}}\mathbf{h} = \mathbf{0}.
\end{equation*}
Similarly, we write the first two penalty terms in \eqref{interfacesat2} as
\begin{align*}
\tau_{21}g\mathbf{M}^{-1}\widetilde{\mathbf{B}}\mathbf{h} + \tau_{22}U\mathbf{M}^{-1}\widetilde{\mathbf{B}}\mathbf{u} &= \frac{1}{2}\mathbf{M}^{-1}\widetilde{\mathbf{B}}(g\mathbf{h} + U\mathbf{u}) \\
&= -\widetilde{\mathbf{D}}(g\mathbf{h} + U\mathbf{u}) + \mathbf{D}(g\mathbf{h} + U\mathbf{u}),
\end{align*}
the penalty terms involving $\gamma_{ij}$ as
\begin{align*}
&\gamma_{21}H^2\mathbf{D}\mathbf{M}^{-1}\widetilde{\mathbf{B}}\frac{d\mathbf{u}}{dt} +\gamma_{22}H^2\mathbf{M}^{-1}\widetilde{\mathbf{B}}\mathbf{D}\frac{d\mathbf{u}}{dt}
+\gamma_{23}H^2\mathbf{M}^{-1}\widetilde{\mathbf{B}}\mathbf{M}^{-1}\widetilde{\mathbf{B}}\frac{d\mathbf{u}}{dt} \\ &= 
\frac{H^2}{3}\left(-\frac{1}{2}\mathbf{D}\mathbf{M}^{-1}\widetilde{\mathbf{B}}-\frac{1}{2}\mathbf{M}^{-1}\widetilde{\mathbf{B}}\mathbf{D}+\frac{1}{4}\left(\mathbf{M}^{-1}\widetilde{\mathbf{B}}\right)^2 \right)\frac{d\mathbf{u}}{dt} \\
&= \frac{H^2}{3}\widetilde{\mathbf{D}}\frac{d\mathbf{u}}{dt} - \frac{H^2}{3}\mathbf{D}^2\frac{d\mathbf{u}}{dt},
\end{align*}
and the remaining penalty terms as 
\begin{align*}
    &\sigma_{21}H^2U \mathbf{D}^2\mathbf{M}^{-1}\widetilde{\mathbf{B}}\mathbf{u} + \sigma_{22}H^2U \mathbf{D}\mathbf{M}^{-1}\widetilde{\mathbf{B}} \mathbf{D}\mathbf{u} + \sigma_{23}H^2U \mathbf{D}\mathbf{M}^{-1}\widetilde{\mathbf{B}}\mathbf{M}^{-1}\widetilde{\mathbf{B}}\mathbf{u} \nonumber \\
    &+\sigma_{24}H^2U\mathbf{M}^{-1}\widetilde{\mathbf{B}}\mathbf{D}^2\mathbf{u} + \sigma_{25}H^2U\mathbf{M}^{-1}\widetilde{\mathbf{B}}\mathbf{D}\mathbf{M}^{-1}\widetilde{\mathbf{B}}\mathbf{u} + \sigma_{26}H^2U\mathbf{M}^{-1}\widetilde{\mathbf{B}}\mathbf{M}^{-1}\widetilde{\mathbf{B}}\mathbf{D}\mathbf{u} \nonumber \\
    &+\sigma_{27}H^2U\mathbf{M}^{-1}\widetilde{\mathbf{B}}\mathbf{M}^{-1}\widetilde{\mathbf{B}}\mathbf{M}^{-1}\widetilde{\mathbf{B}}\mathbf{u} \\
    &= \frac{H^2U}{3}\bigg(-\frac{1}{2}\mathbf{D}^2\mathbf{M}^{-1}\widetilde{\mathbf{B}} - \frac{1}{2}\mathbf{D}\mathbf{M}^{-1}\widetilde{\mathbf{B}} \mathbf{D} + \frac{1}{4}\mathbf{D}\left(\mathbf{M}^{-1}\widetilde{\mathbf{B}}\right)^2 - \frac{1}{2}\mathbf{M}^{-1}\widetilde{\mathbf{B}}\mathbf{D}^2\\ 
    &\quad\quad\quad\quad\quad+\frac{1}{4}\mathbf{M}^{-1}\widetilde{\mathbf{B}}\mathbf{D}\mathbf{M}^{-1}\widetilde{\mathbf{B}} + \frac{1}{4}\left(\mathbf{M}^{-1}\widetilde{\mathbf{B}}\right)^2\mathbf{D} - \frac{1}{8}\left( \mathbf{M}^{-1}\widetilde{\mathbf{B}}\right)^3 \bigg)\mathbf{u} \\
    &=\frac{H^2U}{3}\widetilde{\mathbf{D}}^3\mathbf{u} - \frac{H^2U}{3}\mathbf{D}^3\mathbf{u}.
\end{align*}
Utilizing these equations, \eqref{interfacesat2} becomes
\begin{equation*}
    \frac{d\mathbf{u}}{dt} + \widetilde{\mathbf{D}}\left(g\mathbf{h} + U\mathbf{u} - \frac{H^2U}{3}\widetilde{\mathbf{D}}^2\mathbf{u} - \frac{H^2}{3}\widetilde{\mathbf{D}}\left( \frac{d\mathbf{u}}{dt}\right)\right) +  \alpha_u  \mathbf{M}^{-1}\widetilde{\mathbf{B}}^{\top} \widetilde{\mathbf{B}}\mathbf{u} = \mathbf{0}.
\end{equation*}

\end{document}